\input ./style/arxiv-vmsta.cfg
\documentclass[numbers,compress,v1.0.1]{vmsta}

\usepackage{dcolumn}
\usepackage{mathrsfs}
\usepackage[mathcal]{euscript}
\usepackage{bbm}
\usepackage{bbold}

\volume{4}
\issue{4}
\pubyear{2017}
\firstpage{285}
\lastpage{314}
\doi{10.15559/17-VMSTA88}

\startlocaldefs
\newcommand{\lleft}{\left}
\newcommand{\rright}{\right}
\allowdisplaybreaks

\newtheorem{thm}{Theorem}
\newtheorem{lemma}{Lemma}
\newtheorem{cor}{Corollary}
\newtheorem{prop}{Proposition}

\theoremstyle{definition}
\newtheorem{defin}{Definition}
\newtheorem{example}{Example}
\newtheorem{remark}{Remark}

\newcommand{\idf}{\mathsf{J}} 
\newcommand{\iqf}{\mathsf{K}} 
\newcommand{\EE}{\mathsf{E}}
\newcommand{\PP}{\mathsf{P}}
\newcommand{\DD}{\mathsf{D}}
\newcommand{\QQ}{\mathsf{Q}}
\newcommand{\bb}{\mathsf{b}}
\newcommand{\rr}{\mathsf{r}}
\newcommand{\cF}{{\mathscr{F}}}
\newcommand{\bbR}{\mathbb{R}} 
\newcommand{\E}{\mathbb{E}}
\newcommand{\gN}{\mathfrak{N}}

\endlocaldefs

\begin{document}

\begin{frontmatter}

\title{Integrated quantile functions: properties and applications}

\author[a,b,c]{\inits{A.A.}\fnm{Alexander A.}\snm{Gushchin}\corref{cor1}}\email{gushchin@mi.ras.ru}
\cortext[cor1]{Corresponding author.}
\author[b]{\inits{D.A.}\fnm{Dmitriy A.}\snm{Borzykh}}\email{dborzykh@hse.ru}
\address[a]{{Steklov Mathematical Institute}, Gubkina 8,
119991 Moscow, {Russia}}
\address[b]{{National Research University Higher School
of Economics},\\
Myasnitskaya 20, 101000 Moscow, {Russia}}
\address[c]{{Lomonosov Moscow State University},\\
Leninskie gory, 119991
Moscow, {Russia}}

\markboth{A.A. Gushchin, D.A. Borzykh}{Integrated quantile functions:
properties and applications}

\begin{abstract}
In this paper we provide a systematic exposition of basic properties of
integrated distribution and quantile functions. We define these transforms
in such a way that they characterize any probability distribution on the
real line and are Fenchel conjugates of each other. We show that uniform
integrability, weak convergence and tightness admit a convenient
characterization in terms of integrated quantile functions. As an
application we demonstrate how some basic results of the theory of
comparison of binary statistical experiments can be deduced using
integrated quantile functions. Finally, we extend the area of application
of the Chacon--Walsh construction in the Skorokhod embedding problem.
\end{abstract}

\received{9 August 2017}
\revised{6 November 2017}
\accepted{8 November 2017}
\publishedonline{8 December 2017}
\end{frontmatter}

\section{Introduction}\label{s:sec1}

Integrated distribution and quantile functions or simple
transformations of them play an important role in probability theory,
mathematical statistics, and their applications such as insurance,
finance, economics etc. They frequently appear in the literature, often
under different names. Moreover, in many occasions they are defined
under additional assumption of integrability of a random variable or at
least integrability of the positive or the negative part of a variable.
Let us point out only few references. For a random variable $X$, let
$F_X$ be the distribution function of $X$ and $q_X$ any quantile
function of $X$. Examples of integrated distribution functions or their
simple modifications are:
\begin{itemize}
\item
The function
\begin{equation}
\label{gs76fghvxaf3} \varPsi_X(x) := \int_{-\infty}^{x}F_X(t)
\,dt , \quad x \in\mathbb{R},
\end{equation}
considered in \cite{FollmerSchied2011}.
\item
The integrated survival function
\begin{equation}
\label{jnyg7hxz7zx6} H_X(x) : = \int_{x}^{+\infty}
\bigl(1 - F_X(t)\bigr)\,dt , \quad x \in\mathbb{R},
\end{equation}
of $X$, also called the stop-loss transform, see, e.g., \citep
{Muller1996, MullerStoyan2002}.
\item
The potential function
\begin{equation}
\label{ew6ghzjgaj} U_X(x) : = -\EE|x - X| {, \quad x \in
\mathbb{R},}
\end{equation}
of $X$, see, e.g., \cite{ChaconWalsh1976,Cox2008}.
\end{itemize}
These transforms characterize the distribution of $X$ only if the
expectations $\EE X^{-}$, or $\EE X^{+}$, or $\EE|X|$ respectively are
finite; otherwise, the transforms equal $+\infty$ or $-\infty$
identically and do not allow to identify the distribution of $X$.

The examples of integrated quantile functions or their simple
transformations are:
\begin{itemize}
\item
The absolute Lorenz curve
\begin{equation}
\label{e6gfshf65gfxxxs} \operatorname{AL}_X(u) := \int_{0}^{u}q_X(s)
\,ds , \quad u \in[0,\,1],
\end{equation}
see, e.g., \cite{OgryczakRuszczynske2002} and the references therein.
\item
The Hardy--Littlewood maximal function
\begin{equation}
\label{thsr5gajytghaj} \operatorname{HL}_X(u) := \frac{1}{1-u}\int
_{u}^{1}q_X(s)\,ds , \quad u
\in[0,\,1),
\end{equation}
of $X$, see \cite{HardyLittlewood1930}.
\item
The Conditional Value at Risk
\begin{equation}
\label{urfvzbtrg11a111vyt} \operatorname{CV@R}_X(u) := \frac
{1}{u}\int
_{0}^{u}q_X(s)\,ds , \quad u
\in(0,\,1],
\end{equation}
see, e.g., \citep{RockafellarUryasev2000, RockafellarUryasev2002}, also
called the Average Value at Risk \cite{FollmerSchied2011}, or the
expected shortfall, or the expected tail loss.
\end{itemize}
Again, these transforms characterize the distribution of $X$ if either
$\EE[X^{-}] < \infty$ or $\EE[X^{+}] < \infty$, otherwise, they are
equal to $+\infty$ or $-\infty$ identically.

The main goal of this paper is a systematic exposition of basic
properties of integrated distribution and quantile functions.
In particular, we define the integrated distribution and quantile
functions for any random variable $X$ in such a way that each one of
these functions determines uniquely the distribution of $X$. Further,
we show that such important notions of probability theory as uniform
integrability, weak convergence and tightness can be characterized in
terms of integrated quantile functions (see Section~\ref{s:sec3}). In
Section~\ref{s:sec4} we show how some basic results of the theory of
comparison of binary statistical experiments can be deduced using our
results in previous two sections. Finally, in Section~\ref{s:sec5} we
extend the area of application of the Chacon--Walsh construction in the
Skorokhod embedding problem with the help of integrated quantile functions.

One of the key points of our approach is that we define integrated
distribution and quantile functions as Fenchel conjugates of each
other. This is due to the fact that their derivatives, distribution
functions and quantile functions, are generalized inverses (see, e.\,
g., \citep{EmbrechtsHofert2013, FollmerSchied2011}). This convex
duality result can be found in \cite{OgryczakRuszczynske2002} and
\cite
[Lemma~A.26]{FollmerSchied2011}, and constitutes implicitly one of two
main results in \citep{RockafellarUryasev2000, RockafellarUryasev2002}.

Let us note that we consider only univariate distributions in this
paper. However, it is reasonable to mention a possible generalization
to the multidimensional case based on ideas from optimal transport. The
integrated quantile function of a random variable $X$, as it is defined
in our paper, is a convex function whose gradient pushes forward the
uniform distribution on $(0,\,1)$ into the distribution of $X$;
moreover, the integrated distribution function is the Fenchel transform
of the integrated quantile function and its gradient pushes forward the
distribution of $X$ into the uniform distribution on $(0,\, 1)$ if the
distribution of $X$ is continuous. It the multidimensional case the
existence of such functions follows from the McCann theorem \cite
{McCann1995}. Namely, if $\mu$ is the distribution on $\mathbb{R}^d$,
then there exists a (unique up to an additive constant) convex function
$V$ whose gradient pushes forward the uniform distribution on the unit
cube (or, say, the unit ball) in $\mathbb{R}^d$ into $\mu$.
Additionally, if $\mu$ vanishes on Borel subsets of Hausdorff dimension
$d-1$, then the Fenchel transform $V^{*}$ of $V$ pushes forward $\mu$
to the corresponding uniform distribution. We refer to \cite
{CarlierChernozhukovGalichon2016,ChernozhukovGalichonHallin2017,FaugerasRuschendorf2017}
and
\cite{Hallin2017} for recent advances in this area.

It is more convenient for us to speak about random variables rather
than distributions. However, if a probability space is not specified,
the symbols $\PP$ and $\EE$ for probability and expectation enter into
consideration only via distributions of random variables and may refer
to different probability spaces. This allows us to replace occasionally
random variables by their distributions in the notation.

For the reader's convenience, we recall some terminology and elementary
facts concerning convex functions of one real variable. A convex function
$f\colon\bbR\to\bbR\cup\{+\infty\}$ is proper if its effective domain
\[
\operatorname{dom} f := \bigl\{x\in\bbR\colon f(x)<+\infty\bigr\}
\]
is not empty. The subdifferential $\partial f(x)$ of~$f$ at a point $x$
is defined by
\[
\partial f(x)= \bigl\{u\in\bbR\colon f(y)\geq f(x)+ u(y-x) \mbox{
for every } y
\in\bbR\bigr\}.
\]
If $f$ is a proper convex function and $x$ is an interior point of
$\operatorname{dom} f$, then $\partial f(x)=[f'_-(x),f'_+(x)]$, where
$f'_-(x)$ and $f'_+(x)$ are the left and right derivatives of $f$ at
$x$ respectively.
The conjugate of~$f$, or the Fenchel transform, is the function~$f^*$
on $\bbR$ defined by
\[
f^*(u)=\sup_{x\in\bbR}\, \bigl[xu-f(x)\bigr].
\]
The conjugate function is lower semicontinuous and convex. The
Fenchel--Moreau theorem says that if $f$ is a proper lower
semicontinuous convex function, then $f$ is the conjugate of~$f^*$,~i.e.
\[
f(x)=\sup_{u\in\bbR} \bigl[xu-f^*(u) \bigr],\quad x\in\bbR;
\]
moreover, for $x,u\in\bbR$,
\[
u\in\partial f(x)\quad\Longleftrightarrow\quad x\in\partial
f^*(u)\quad\ \Longleftrightarrow\quad
f(x)+f^*(u) = xu.
\]

\section{Integrated distribution and quantile functions: definitions
and main properties}\label{s:sec2}

\subsection{Definition and properties of integrated distribution
functions}\label{ss:ssec21}
The distribution function $F_X$ of a random variable $X$ given on a
probability space $(\varOmega, \, \mathscr{F}, \, \PP)$ is defined by
$F_X(x) = \PP(X \leq x)$, $x \in\mathbb{R}$. Since $F_X$ is bounded,
for any choice of $x_0 \in\mathbb{R}$, the integral $\int
_{x_0}^{x}F_X(t)\,dt$ is defined and finite for all $x \in\mathbb
{R}$.\footnote{Throughout the paper, if $b < a$, by convention $\int
_{a}^{b}f(x)\,dx := -\int_{b}^{a}f(x)\,dx$.} In contrast to this case,
the function $\varPsi_X$ in (\ref{gs76fghvxaf3}) corresponding to the
choice $x_0 = -\infty$, takes value $+\infty$ identically if $\EE
[X^{-}] = \infty$.

\begin{defin}
The \textit{integrated distribution function} of a random variable $X$
is defined by
\[
\idf_X(x) := \int_{0}^{x}F_X(t)
\,dt , \quad x \in\mathbb{R} .
\]
\end{defin}

\begin{thm}\label{ystwqe52dgh}
An integrated distribution function $\idf_X$ has the following
properties\/{\rm:}
\begin{enumerate}
\item[\emph{(i)}] $\idf_X(0) = 0$.
\item[\emph{(ii)}] $\idf_X$ is convex{\rm,} increasing and finite
everywhere on $\mathbb{R}$.
\item[\emph{(iii)}] for $a < b$,
\begin{equation}
\label{jsd7hgaarg1f} \idf_X(b) - \idf_X(a) = \EE
\bigl[(b-X)^{+} - (X-a)^{-}\bigr],
\end{equation}
in particular{\rm,} for $x \in\mathbb{R}$,
\begin{equation}
\label{ghs1x2tbn615g1} \idf_X(x) = \EE\bigl[(x - X)^{+} -
X^{-}\bigr] = \EE\bigl[(X - x)^{+} - X^{+} + x
\bigr] \text{.}
\end{equation}
\item[\emph{(iv)}] $\lim\limits_{x \rightarrow-\infty}\idf_X(x)
= -\EE
[X^{-}]$ and $\lim\limits_{x \rightarrow+\infty}(x - \idf_X(x)) =
\EE[X^{+}]$.
\item[\emph{(v)}] $\lim\limits_{x \rightarrow-\infty}\tfrac{\idf
_X(x)}{x} = 0$ and $\lim\limits_{x \rightarrow+\infty}\tfrac{\idf
_X(x)}{x} = 1$.
\item[\emph{(vi)}] The subdifferential of $\idf_X$ satisfies
\begin{equation}
\label{f6hxbzc6tgasvn3v} \partial\idf_X(x) = \bigl[F_X(x-0), \,
F_X(x)\bigr], \quad x \in\mathbb{R},
\end{equation}
in particular{\rm,} $(\idf_X)'_{-}(x) = F_X(x-0)$ and $(\idf
_X)'_{+}(x) = F_X(x)$.
\item[\emph{(vii)}] $\idf_{-X}(x) = x + \idf_{X}(-x)$ for all $x
\in
\mathbb{R}$.
\end{enumerate}
\end{thm}

It is clear from (vi), that the integrated distribution function
uniquely determines the distribution.

\begin{proof}
It is evident that (i) holds and $\idf_X$ is finite and increasing. For
$a < b$, we have
\begin{equation}
\label{sjy6wghgsyg} F_X(a) (b-a) \leq\idf_X(b) -
\idf_X(a) = \int_{a}^{b}F_X(t)
\,dt \leq F_X(b-0) (b-a) \text{.}
\end{equation}
It follows that, for any $x, \, y \in\mathbb{R}$,
\[
\idf_X(y) \geq\idf_X(x) + p (y - x) \text{,}
\]
if $p \in[F_X(x-0), \, F_X(x)]$. Now the convexity of $\idf_X$
follows, which, in turn, implies (vi).

Next, by Fubini's theorem, for $a < b$,
\begin{align*}
\int_{a}^{b}F_X(t)\,dt &= \int
_{a}^{b}\EE[\mathbb{1}_{\{X \leq t\}}]\, dt = \EE
\Biggl[\int_{a}^{b}\mathbb{1}_{\{X \leq t\}}\,dt
\Biggr]
\\
&= \EE\bigl[(b-X)^{+} - (X-a)^{-}\bigr] \text{.}
\end{align*}
Thus, we have proved (\ref{jsd7hgaarg1f}). The second equality in
(\ref
{ghs1x2tbn615g1}) is trivial, and the first one follows from (\ref
{jsd7hgaarg1f}) if we put $a = 0$ or $b = 0$ depending on the sign of $x$.

Let us prove (iv). The function $(x - X)^{+} - X^{-}$ is increasing in
$x$, hence $\EE[(x - X)^{+} - X^{-}] \to-\EE[X^{-}]$ as $x
\rightarrow
-\infty$ by the monotone convergence theorem. This proves the first
equality in (iv). Similarly, $x - (x - X)^{+} + X^{-}$ is increasing in
$x$, hence $\EE[x - (x - X)^{+} + X^{-}] \to\EE[X^{+}]$ as $x
\rightarrow+\infty$ by the monotone convergence theorem.

Finally, (v) and (vii) follow from (\ref{sjy6wghgsyg}) and (\ref
{ghs1x2tbn615g1}) respectively.
\end{proof}

\setlength{\unitlength}{1mm}

\begin{figure}[t]
\includegraphics{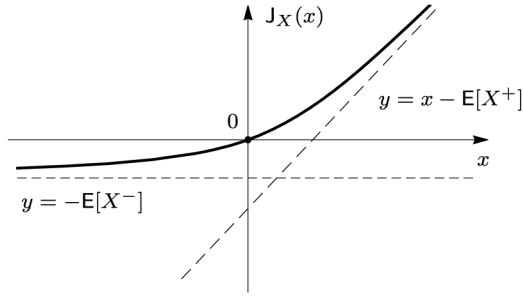}
\caption{A typical graph of an integrated distribution function if the
expectations $\EE[X^{-}]$ and $\EE[X^{+}]$ are finite}\label{f2}
\end{figure}

\begin{cor}\label{c23sjg5kjhgf}
If $X$ is an integrable random variable{\rm,} then{\rm,} for any $x
\in\mathbb{R}$,
\begin{align*}
\varPsi_X(x) &= \idf_X(x) + \EE\bigl[X^{-}
\bigr],
\\
H_X(x) &= \idf_X(x) + \EE\bigl[X^{+}\bigr]
- x,
\\
U_X(x) &= x - \EE|X| - 2\idf_X(x),
\end{align*}
where $\varPsi_X$, $H_X$, and $U_X$ are defined in\/ {\rm(\ref
{gs76fghvxaf3})--(\ref{ew6ghzjgaj}),} in particular{\rm,}
\[
\varPsi_X(x) + H_X(x) = - U_X(x) \text{.}
\]
\end{cor}

\begin{thm}\label{v6gsh67htyfad}
If $J \colon\mathbb{R} \to\mathbb{R}$, $J(0) = 0$, is a convex
function satisfying
\[
\lim\limits
_{x \rightarrow-\infty}\frac{J(x)}{x} = 0 \quad\text{and}\quad
\lim\limits
_{x \rightarrow+\infty}
\frac{J(x)}{x} = 1,
\]
then there exists on some probability space a random variable $X$ for
which $\idf_X = J$.
\end{thm}

\begin{proof}
Since $J$ is convex and finite everywhere on the line, it has the
right-hand derivative at each point, and $J(x) = \int
_{0}^{x}J'_{+}(t)\,
dt$. Moreover, similarly to the proof of (v) in Theorem \ref
{ystwqe52dgh}, $\lim_{x \rightarrow-\infty}J'_{+}(x) = \lim_{x
\rightarrow-\infty}\tfrac{J(x)}{x} = 0$ and $\lim_{x \rightarrow
+\infty}J'_{+}(x) = \lim_{x \rightarrow+\infty}\tfrac{J(x)}{x} = 1$.
Put $F(x) := J'_{+}(x)$. Due to convexity of $J$, $F$ is an increasing
and right-continuous function. So we can conclude, that $F$ is the
distribution function of some random variable $X$ and $\idf_X = J$.
\end{proof}

\subsection{Definition and properties of integrated quantile
functions}\label{ss:ssec22}

We call every function $q_X \colon(0,\,1) \to\mathbb{R}$ satisfying
\[
F_X\bigl(q_X(u) - 0\bigr) \leq u \leq F_X
\bigl(q_X(u)\bigr) ,\quad u \in(0,\,1),
\]
a \textit{quantile function} of a random variable $X$.
The functions $q_X^{L}$ and $q_X^{R}$ defined by
\begin{align*}
&q_X^{L}(u) := \inf\bigl\{x \in\mathbb{R} \colon
F_X(x) \geq u\bigr\} ,\\
&q_X^{R}(u) := \inf\bigl\{x \in\mathbb{R} \colon
F_X(x) > u\bigr\} ,
\end{align*}
are called the \textit{lower (left)} and \textit{upper (right) quantile
functions} of $X$. Of course, the lower and upper quantile functions of
$X$ are quantile functions of $X$, and we always have
\[
q_X^{L}(u) \leq q_X(u) \leq
q_X^{R}(u) , \quad u \in(0,\,1),
\]
for any quantile function $q_X$.

It follows directly from the definitions that, for any $x \in\mathbb
{R}$ and $u \in(0,\,1)$,
\begin{align}
\label{y7f7zx9876gj7ccs2HHH} q_X^{L}(u) &\leq x \quad\Leftrightarrow
\quad u
\leq F_X(x) \text{,}\\
\label{n98j98bxc7aHHH} q_X^{R}(u) &\geq x \quad\Leftrightarrow\quad u
\geq F_X(x-0) \text{.}
\end{align}
See, e.\,g., \citep{EmbrechtsHofert2013, FollmerSchied2011} for more
information on quantile functions (generalized inverses).

\begin{defin}\label{d:iqf}
The Fenchel transform of the integrated distribution function of a
random variable $X$
\begin{equation}
\label{kusd73hjds} \iqf_X(u) = \sup_{x \in\mathbb{R}}\bigl\{xu -
\idf_X(x)\bigr\} , \quad u \in\mathbb{R},
\end{equation}
is called the \textit{integrated quantile function} of $X$.
\end{defin}

This definition is motivated by the fact mentioned in the introduction,
that a function whose derivative is a quantile function must coincide
with the Fenchel transform of $\idf_X$ up to an additive constant. The
next theorem clarifies this point.

\begin{thm}\label{t:iqf}
An integrated quantile function $\iqf_X$ has the following properties\/
{\rm:}
\begin{enumerate}
\item[\emph{(i)}] The function $\iqf_X$ is convex and lower
semicontinuous. It takes finite values on $(0,\,1)$ and equals $+\infty
$ outside $[0,\,1]$.
\item[\emph{(ii)}] The Fenchel transform of $\iqf_X$ is $\idf_X$,
i.\,
e. for any $x \in\mathbb{R}$,
\begin{equation}
\label{d78hjxxxxx76shs} \idf_X(x) = \sup_{u \in\mathbb{R}}\bigl\{
xu -
\iqf_X(u)\bigr\} \text{.}
\end{equation}
\item[\emph{(iii)}] $\min_{u \in\mathbb{R}}\iqf_X(u) = 0$, $\{u
\in
\mathbb{R} \colon\iqf_X(u) = 0\} = [F_X(0-0), \, F_X(0)]$.
\item[\emph{(iv)}] for every $u \in[0,\,1]$,
\begin{equation}
\label{jsytshv12dvtg43} \iqf_X(u) = \int_{u_0}^{u}q_X(s)
\,ds,
\end{equation}
where $u_0$ is any zero of $\iqf_X$.
\item[\emph{(v)}] $\iqf_X(0) = \EE[X^{-}]$ and $\iqf_X(1) = \EE[X^{+}]$.
\item[\emph{(vi)}] The subdifferential of $\iqf_X$ satisfies
\begin{equation}
\label{g7ggg7gyhsj23} \partial\iqf_X(u) = \bigl[q_X^{L}(u),
\, q_X^{R}(u)\bigr], \quad u \in(0,\,1),
\end{equation}
in particular{\rm,} $(\iqf_X)'_{-}(u) = q_X^{L}(u)$ and $(\iqf
_X)'_{+}(x) = q_X^{R}(u)$.
\item[\emph{(vii)}] $\iqf_{-X}(u) = \iqf_{X}(1 - u)$ for all $u \in
[0,\,1]$.
\end{enumerate}
\end{thm}

It is clear from (ii) and the similar remark after Theorem~\ref
{ystwqe52dgh} that the integrated quantile function uniquely determines
the distribution.

\begin{figure}[t]
\includegraphics{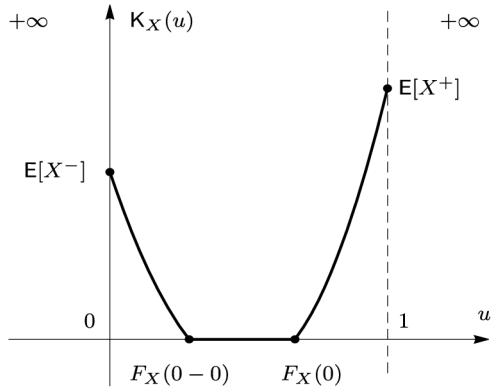}
\caption{A typical graph of an integrated quantile function}\label
{hb134tr12aa}
\end{figure}

\begin{proof}
Since $\idf_X$ is a proper convex continuous function, it follows from
the definition of $\iqf_X$ and the Fenchel--Moreau theorem that $\iqf
_X$ is convex and lower semicontinuous, (\ref{d78hjxxxxx76shs}) holds,
and for all $x,\, u \in\mathbb{R}$
\begin{equation}
\label{jsyajtv122sd} u \in\partial\idf_X(x) \quad\Leftrightarrow
\quad x \in
\partial\iqf_X(u) \quad\Leftrightarrow\quad F_X(x - 0) \leq
u \leq F_X(x) \text{,}
\end{equation}
where the last equivalence follows from (\ref{f6hxbzc6tgasvn3v}). In
particular,
\[
\partial\iqf_X(u) = \varnothing,\quad \text{if}\ u \notin[0,
\,1],
\]
and, for $u \in(0,\,1)$,
\[
x\in\partial\iqf_X(u) \quad\Leftrightarrow\quad q_X^{L}(u)
\leq x \leq q_X^{R}(u) \text{,}
\]
due to (\ref{y7f7zx9876gj7ccs2HHH}) and (\ref{n98j98bxc7aHHH}). Thus,
we have proved (i), (ii) and (vi).\vadjust{\goodbreak}

Putting $x = 0$ in (\ref{d78hjxxxxx76shs}) and (\ref{jsyajtv122sd}), we
get $\inf_{u \in\mathbb{R}}\iqf_X(u) = 0$ and this infimum is attained
at $u$ if and only if $u \in[F_X(0-0), \, F_X(0)]$. This constitutes
assertion~(iii). Now (iv) follows from preceding statements.

Statement (v) follows from the definition of $\iqf_X$ and Theorem~\ref
{ystwqe52dgh}~(iv):
\begin{align*}
\iqf_X(0) &= -\inf_{x \in\mathbb{R}}\idf_X(x) = -
\lim_{x \rightarrow
-\infty}\idf_X(x) = \EE\bigl[X^{-}
\bigr] \text{,}\\
\iqf_X(1) &= \sup_{x \in\mathbb{R}} \bigl(x -
\idf_X(x) \bigr) = \lim_{x
\rightarrow+\infty} \bigl(x -
\idf_X(x) \bigr) = \EE\bigl[X^{+}\bigr] \text{.}
\end{align*}

Finally, (vii) follows from the definition of $\iqf_X$ and
Theorem~\ref
{ystwqe52dgh}~(vii).
\end{proof}

\begin{cor}\label{bsat65gh2hjs}
For any random variable $X$,
\[
\iqf_X\bigl(F_X(x)\bigr) = x F_X(x) -
\idf_X(x), \quad x \in\mathbb{R}.
\]
\end{cor}

\begin{proof}
Put $u := F_X(x)$ and $g(y) := \idf_X(y) - yu$, $y \in\mathbb{R}$.
According to (\ref{f6hxbzc6tgasvn3v}), $\partial g(y) = [F_X(y-0) - u,
\, F_X(y) - u]$, in particular, $0 \in\partial g(y)$ if $y = x$. This
means that the function $g$ attains its minimum at $x$ and, hence, we
have $\iqf_X(u) = \sup_{y \in\mathbb{R}}\{yu - \idf_X(y)\} = xu -
\idf_X(x)$.
\end{proof}

\begin{thm}\label{hs6gsvgaxs3}
If a convex lower semicontinuous function $K \colon\mathbb{R} \to
\mathbb{R}_{+} \cup\{+\infty\}$ satisfies
\[
(0,\,1) \subseteq\operatorname{dom} K \subseteq[0,\,1]
\]
and there is $u_0 \in[0,\,1]$ such that $K(u_0) = 0$, then there
exists on some probability space a random variable $X$ for which $\iqf
_X = K$.
\end{thm}

\begin{proof}
Under our assumptions
\[
K(u) = \int_{u_0}^{u}q(s)\,ds , \quad u
\in[0,\,1],
\]
where $q(u) = K'_{-}(u)$, $u \in(0,\,1)$, is increasing and left
continuous. Let us define a probability space $(\varOmega, \, \cF,
\PP)$
as follows: $\varOmega= (0,\,1)$, $\cF$ is the Borel $\sigma$-field and
$\PP$ is the Lebesgue measure. Put $X(\omega) := q(\omega)$. Now if
$G(x) := \inf\{u \in(0,\,1) \colon q(u) > x\}$, then it is easy to
verify that $q(u) \leq x \; \Leftrightarrow\; G(x) \geq u$, cf. (\ref
{y7f7zx9876gj7ccs2HHH}). It follows that $G$ is the distribution
function of $X$ and, hence, $q = q_X^{L}$ on $(0,\,1)$. This means that
the left-hand derivative of $K$ and $\iqf_X$ coincide on $(0,\,1)$. In
addition, their minimums over this interval are equal to zero.
Therefore, $K = \iqf_X$ on (0,\,1) and, hence, everywhere on $\mathbb{R}$.
\end{proof}

\begin{remark}
An alternative way to prove Theorem~\ref{hs6gsvgaxs3} is to introduce
the Fenchel transform $J$ of $K$ and to show that $J$ satisfies the
assumptions of Theorem \ref{v6gsh67htyfad}. However, our proof yields
not only a characterization statement of Theorem~\ref{hs6gsvgaxs3} but
also an explicit representation of a random variable with a given
integrated quantile function.
Of course, this representation (namely, of a random variable with given
distribution as its quantile function with respect to the Lebesgue
measure on $(0,\,1)$) is well known.
\end{remark}

It is convenient to introduce \textit{shifted integrated quantile functions}:
\begin{align*}
&\iqf_X^{[0]}(u) := \iqf_X(u) -
\iqf_X(0) , \quad u \in[0,\,1], \qquad\text{if}\,\
\iqf_X(0) = \EE\bigl[X^{-}\bigr] < \infty, \\
&\iqf_X^{[1]}(u) := \iqf_X(u) -
\iqf_X(1) , \quad u \in[0,\,1], \qquad\text{if}\,\
\iqf_X(1) = \EE\bigl[X^{+}\bigr] < \infty.
\end{align*}

Now we can express the functions defined in (\ref
{e6gfshf65gfxxxs})--(\ref{urfvzbtrg11a111vyt}) in terms of shifted
integrated quantile functions. If $\EE[X^{-}] < \infty$, then the
absolute Lorenz curve coincides with $\iqf_X^{[0]}$:
\[
\operatorname{AL}_X(u) = \iqf_X^{[0]}(u)
, \quad u \in[0,\,1].
\]
Since $\varPsi_X$ is obtained from $\idf_X$ by adding a constant $\EE
[X^{-}] = \iqf_X(0)$ by Corollary~\ref{c23sjg5kjhgf}, the absolute
Lorenz curve is the Fenchel transform of $\varPsi_X$:
\[
\operatorname{AL}_X(u) = \sup_{x \in\mathbb{R}}\bigl\{xu -
\varPsi_X(x)\bigr\} , \quad u \in[0,\,1].
\]
The Conditional Value at Risk satisfies
\[
\operatorname{CV@R}_X(u) = \frac{1}{u}\iqf_X^{[0]}(u)
= \sup_{x \in
\mathbb{R}}\bigl\{x - \varPsi_X(x)/u\bigr\}
, \quad u \in(0,\,1].
\]
The Hardy--Littlewood maximal function satisfies
\[
\operatorname{HL}_X(u) = \frac{1}{u-1}\iqf_X^{[1]}(u)
= \frac
{1}{u-1}\iqf_{-X}^{[0]}(1-u) = -
\operatorname{CV@R}_{-X}(1-u) , \quad u \in[0,\,1).
\]

\subsection{Convex orders}

Let us recall the definitions of convex orders in the univariate case.

For an arbitrary function $\psi\colon\mathbb{R} \to\mathbb{R}_{+}$,
define $C_{\psi}$ as the space of all continuous functions $f \colon
\mathbb{R} \to\mathbb{R}$ such that
\[
\sup_{x \in\mathbb{R}}\frac{|f(x)|}{1 + \psi(x)} < \infty\text{.}
\]

Let $X$ and $Y$ be random variables. We say that
\begin{itemize}
\item\textit{$X$ is less than $Y$ in convex order} ($X \leq_{cx} Y$)
if $\EE|X| < \infty$, $\EE|Y| < \infty$, and $\EE[\varphi(X)]
\leq\EE
[\varphi(Y)]$ for any convex function $\varphi\in C_{|x|}$;
\item\textit{$X$ is less than $Y$ in increasing convex order} ($X
\leq
_{icx} Y$) if $\EE[X^{+}] < \infty$, $\EE[Y^{+}] < \infty$, and
$\EE
[\varphi(X)] \leq\EE[\varphi(Y)]$ for any increasing convex function
$\varphi\in C_{x^{+}}$;
\item\textit{$X$ is less than $Y$ in decreasing convex order} ($X
\leq
_{decx} Y$) if $\EE[X^{-}] < \infty$, $\EE[Y^{-}] < \infty$, and
$\EE
[\varphi(X)] \leq\EE[\varphi(Y)]$ for any decreasing convex function
$\varphi\in C_{x^{-}}$.
\end{itemize}

It is trivial that $X \leq_{icx} Y$ if and only if $-X \leq_{decx} -
Y$. Also it is easy to see that $X \leq_{cx} Y$ if and only if $X \leq
_{icx} Y$ and $X \leq_{decx} Y$.

The following theorem is well known. We provide its proof which reduces
to the duality between integrated distribution and quantile functions.

\begin{thm}\label{jhdt67thdbba}
Let $X$ and $Y$ be random variables.
\begin{itemize}
\item[\emph{(i)}] If $\EE|X| < \infty$, $\EE|Y| < \infty$, then the
following statements are equivalent\/{\rm:}
\begin{itemize}
\item[\emph{(a)}] $X \leq_{cx} Y$;
\item[\emph{(b)}] $\iqf_X^{[1]}(u) \geq\iqf_Y^{[1]}(u)$ for all $u
\in
[0,\,1]$ and $\iqf_X^{[1]}(0) = \iqf_Y^{[1]}(0)$;
\item[\emph{(c)}] $\iqf_X^{[0]}(u) \geq\iqf_Y^{[0]}(u)$ for all $u
\in
[0,\,1]$ and $\iqf_X^{[0]}(1) = \iqf_Y^{[0]}(1)$;
\item[\emph{(d)}] $\iqf_X^{[1]}(u) \geq\iqf_Y^{[1]}(u)$ and $\iqf
_X^{[0]}(u) \geq\iqf_Y^{[0]}(u)$ for all $u \in[0,\,1]$.
\end{itemize}
\item[\emph{(ii)}] $X \leq_{icx} Y$ if and only if $\EE[X^{+}] <
\infty
$, $\EE[Y^{+}] < \infty$, and $\iqf_X^{[1]}(u) \geq\iqf_Y^{[1]}(u)$
for all $u \in[0,\,1]$;
\item[\emph{(iii)}] $X \leq_{decx} Y$ if and only if $\EE[X^{-}] <
\infty$, $\EE[Y^{-}] < \infty$, and $\iqf_X^{[0]}(u) \geq\iqf
_Y^{[0]}(u)$ for all $u \in[0,\,1]$.
\end{itemize}
\end{thm}

\begin{proof}
First, let us prove (ii). It is well known (see, e.\,g., \cite
{Ruschendorf2013}) that $X \leq_{icx} Y$ if and only if $\EE[X^{+}] <
\infty$, $\EE[Y^{+}] < \infty$, and $\EE[(X - x)^{+}] \leq\EE[(Y -
x)^{+}]$ for all $x \in\mathbb{R}$. Taking (\ref{ghs1x2tbn615g1}) into
account, the last condition can be rewritten as
\[
\idf_X(x) + \EE\bigl[X^{+}\bigr] \leq
\idf_Y(x) + \EE\bigl[Y^{+}\bigr] , \quad x \in
\mathbb{R},
\]
which in turn, is equivalent to
\[
\iqf_X(u) - \EE\bigl[X^{+}\bigr] \geq
\iqf_Y(u) - \EE\bigl[Y^{+}\bigr] , \quad u \in[0,
\,1],
\]
by the definition of integrated quantile function. The claim follows.

Now, (iii) follows from (ii) and the first part of the remark before
Theorem~\ref{jhdt67thdbba}. Now, the second part of this remark shows
equivalence (a) $\Leftrightarrow$ (d) in (i).

Next, the equalities in (b) and (c) are both equivalent to $\EE[X] =
\EE
[Y]$. On the other hand, the inequalities in (d) reduce to $-\EE[X]
\geq-\EE[Y]$ and $\EE[X] \geq\EE[Y]$ for $u = 0$ and $u = 1$
respectively. It follows that (d) implies (b) and (c). Finally, it is
straightforward to check that (b) and (c) are equivalent and, hence,
imply (d).
\end{proof}

Further properties of convex orders see, e.g., in \citep
{MullerStoyan2002, Ruschendorf2013}.

\subsection{Examples}

In this subsection we demonstrate how the developed techniques can be
used to derive two elementary well-known inequalities, see \cite[p.
152]{Feller1971}. This approach allows us to find the distributions at
which the corresponding extrema are attained. So the inequalities
obtained in this way are sharp.

\begin{example}\label{hsag71hgjs}
Let $X$ be a random variable with zero mean and finite variance $\DD(X)
= \sigma^2$. It is required to find a \textit{sharp} upper bound for
the probability $\PP(X \geq t)$, where $t$ is a fixed positive number.

We solve a converse problem. Namely, let $p := \PP(X \geq t)$ be fixed.
Our purpose is to find a sharp lower bound for variances $\DD(X) = \EE
[X^2]$ over all random variables $X$ such that $\EE[X] = 0$, and $\PP(X
\geq t) = p$.
\begin{figure}[t!]
\includegraphics{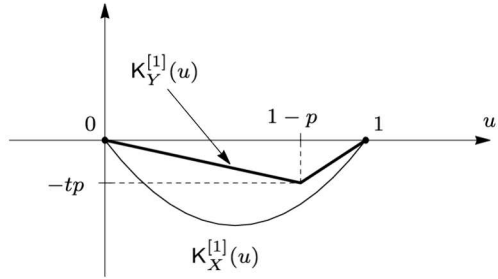}
\caption{Graphs of shifted integrated quantile functions $\iqf_X^{[1]}$
and $\iqf_Y^{[1]}$ in Example~\ref{hsag71hgjs}}\label{f_ekdj73h}
\end{figure}

The above class of distributions has a minimal element with respect to
the convex order. Indeed, let $Y$ be a discrete random variable with
$\PP(Y = t) = p$ and $\PP(Y = - \tfrac{tp}{1 - p}) = 1 - p$. It is
clear that $\EE[Y] = 0$ and $\PP(Y \geq t) = p$. If $X$ is another
random variable with these properties, then $\iqf_X^{[1]}(u) \leq\iqf
_Y^{[1]}(u)$ for all $u \in[0,\,1]$. Indeed, $\iqf_X^{[1]}(0) = \iqf
_Y^{[1]}(0)=0$ and the graph of $\iqf_Y^{[1]}$ consists of two straight
segments, see Fig.~\ref{f_ekdj73h}. Since $\PP(X \geq t) = p$,
$q_{X}^{R}(u) \geq t$ for $u \in[1 - p, \, 1]$. In particular,
\[
\iqf_X^{[1]}(1 - p) = - \int_{1 - p}^{1}q_X^{R}(s)
\,ds \leq-p t = \iqf_Y^{[1]}(1 - p) \text{.}
\]
Due to convexity of integrated quantile functions, this implies $\iqf
_X^{[1]}(u) \leq\iqf_Y^{[1]}(u)$ for all $u \in[0,\,1]$, see
Fig.~\ref
{f_ekdj73h}. Hence, $X \geq_{cx} Y$ by Theorem~\ref{jhdt67thdbba}~(i).
Therefore, $\EE[f(X)] \geq\EE[f(Y)]$ for any convex function $f$. In
particular, $\sigma^2 = \EE[X^2] \geq\EE[Y^2] = \tfrac{t^2 p}{1 - p}$.
Resolving this inequality with respect to $p = \PP(X \geq t)$, we
obtain the required upper bound
\begin{equation}
\label{dks73h32dsz} \PP(X \geq t) \leq\frac{\sigma^2}{\sigma^2 +
t^2} \text{.}
\end{equation}
To show that the estimate in (\ref{dks73h32dsz}) is sharp it is enough
to put $p = \tfrac{\sigma^2}{\sigma^2 + t^2}$ in the definition of a
random variable $Y$ and to check that $\EE[Y^2] = \sigma^2$ and, for $X
= Y$, the equality holds in (\ref{dks73h32dsz}).
\end{example}

\begin{example}\label{nxh5hgs1}
Let $X$ be a strictly positive random variable, i.e. $F_X(0) = 0$, such
that $\EE[X] = 1$ and $\EE[X^2] = b$. It is required to find a
\textit
{sharp} lower bound for the probability $\PP(X > a)$, where $a \in
(0,\,
1)$ is fixed.

We will proceed in the similar way as in the previous example. Namely,
let $p := \PP(X > a)$ be fixed. Our purpose is to find a sharp lower
bound for the second moment $\EE[X^2]$ over all random variables $X$
such that $F_X(0) = 0$, $\EE[X] = 1$ and $\PP(X > a) = p$.

The above class of distributions has a minimal element with respect to
the convex order. Indeed, let $Y$ be a random variable such that $\PP(Y
= a) = 1 - p$ and $\PP(Y = \tfrac{1 - a(1 - p)}{p}) = p$. It is obvious
that $F_Y(0) = 0$, $\EE[Y] = 1$, and $\PP(Y > a) = p$. If $X$ is
another random variable with these properties, then $\iqf_X^{[1]}(u)
\leq\iqf_Y^{[1]}(u)$ for all $u \in[0,\,1]$. Indeed, $\iqf_X^{[1]}(0)
= \iqf_Y^{[1]}(0) = -1$ and the graph of $\iqf_Y^{[1]}$ consists of two
straight segments, see Fig.~\ref{f_akduf763u}. Since $\PP(X \leq a) =
1-p$, $q_{X}^{L}(u) \leq a$ for $u \in[0, \, 1 - p]$. In particular,
\[
\iqf_X^{[1]}(1 - p) = \int_{0}^{1 - p}q_X^{L}(s)
\,ds - \EE[X] \leq a(1-p)-1 = \iqf_Y^{[1]}(1 - p).
\]
Due to convexity of integrated quantile functions, this implies $\iqf
_X^{[1]}(u) \leq\iqf_Y^{[1]}(u)$ for all $u \in[0,\,1]$, see
Fig.~\ref
{f_akduf763u}. Hence, $X \geq_{cx} Y$ by Theorem~\ref
{jhdt67thdbba}~(i). Therefore, $\EE[f(X)] \geq\EE[f(Y)]$ for any
convex function $f$. In particular, $b = \EE[X^2] \geq\EE[Y^2] = a^2
(1 - p) + \tfrac{(1 - a(1 - p))^2}{p^2} p$. Resolving this inequality
with respect to $p = \PP(X > a)$, we obtain the required lower bound
\begin{equation}
\label{j6stg8hjz} \PP(X > a) \geq\frac{(1-a)^2}{b-a(2-a)} \text{.}
\end{equation}
The sharpness of the estimate in (\ref{j6stg8hjz}) follows if we put $p
= \tfrac{(1-a)^2}{b-a(2-a)}$ in the definition of a random variable $Y$
and verify that $\EE[Y^2] = b$ and, for $X = Y$, the equality holds in
(\ref{j6stg8hjz}). Remark that replacing the right-hand side in (\ref
{j6stg8hjz}) by a smaller quantity $\tfrac{(1-a)^2}{b}$, we arrive at
the inequality (7.6) in \cite[p. 152]{Feller1971}.

\begin{figure}[t]
\includegraphics{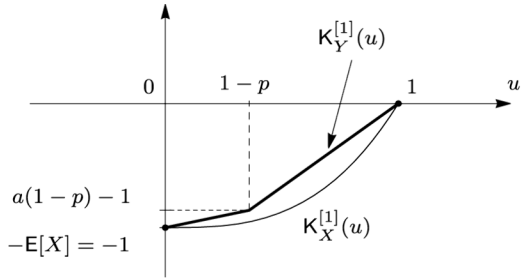}
\caption{Graphs of shifted integrated quantile functions $\iqf_X^{[1]}$
and $\iqf_Y^{[1]}$ in Example~\ref{nxh5hgs1}}\label{f_akduf763u}
\end{figure}
\end{example}

\section{Uniform integrability and weak convergence}\label{s:sec3}

\subsection{Tightness and uniform integrability}
In this subsection we study conditions for tightness and uniform
integrability of a family of random variables in terms of integrated
quantile function. It is a natural question because both tightness and
uniform integrability are characterized in terms of one-dimensional
distributions of these variables.

\begin{thm}\label{t:tightness}
Let $(X_{\alpha})$ be a family of random variables. Then the following
statements are equivalent\/{\rm:}
\begin{itemize}
\item[\emph{(i)}] The family of distributions $\{\operatorname
{Law}(X_{\alpha})\}$ is tight.
\item[\emph{(ii)}] For every $u,\,v\in(0,\,1)$, $\sup_\alpha|\iqf
_{X_{\alpha}}(u) - \iqf_{X_{\alpha}}(v)| < \infty$.
\item[\emph{(iii)}] The family of functions $\{\iqf_{X_{\alpha}}\}$ is
pointwise bounded on $(0,\,1)$.
\item[\emph{(iv)}] The family of functions $\{\iqf_{X_{\alpha}}\}$ is
equicontinuous on every $[a, \, b] \subset(0,\,1)$.
\end{itemize}
\end{thm}

\begin{proof}
(i)~$\Rightarrow$~(iv) Let $[a, \, b] \subset(0,\,1)$. The tightness
condition implies that there is $C > 0$ such that $F_{X_\alpha}(-C) <
a$ and $F_{X_\alpha}(C) \geq b$ for all $\alpha$. Hence, $-C <
q_{X_\alpha}^{L}(u) \leq C$ for all $\alpha$ and $u \in[a, \, b]$.
Thus, the functions $\iqf_{X_\alpha}$ are even uniformly Lipschitz
continuous on $[a, \, b]$.

(ii)~$\Rightarrow$~(iii) Let $c\in(0,\,1)$. Take $a\in(0,\,c)$ and
$b\in
(c,\,1)$. By the assumption, there is $L>0$ such that $\iqf_{X_{\alpha
}}(a) - \iqf_{X_{\alpha}}(c) \leq L$ and $\iqf_{X_{\alpha}}(b) -
\iqf
_{X_{\alpha}}(c) \leq L$ for all $\alpha$. Let $u_{0, \alpha}$ be a
point, where $\iqf_{X_{\alpha}}(u_{0, \alpha}) = 0$. If $\alpha$ is
such that $u_{0, \alpha} < c$, then, by the three chord inequality,
\[
\frac{\iqf_{X_{\alpha}}(c)}{c - u_{0, \alpha}} \leq\frac{\iqf
_{X_{\alpha}}(b) - \iqf_{X_{\alpha}}(c)}{b - c} \text{,}
\]
therefore, $\iqf_{X_{\alpha}}(c) \leq c L / (b - c)$. Similarly, if
$\alpha$ is such that $u_{0, \alpha} > c$, then $\iqf_{X_{\alpha}}(c)
\leq(1-c)L / (c - a)$.

(ii)~$\Rightarrow$~(i) Let $\varepsilon> 0$. By the assumption, there
is $L > 0$ such that $ \iqf_{X_{\alpha}}(\varepsilon) - \iqf
_{X_{\alpha
}}(\varepsilon/2) > -L$ and $\iqf_{X_{\alpha}}(1 - \varepsilon/ 2) -
\iqf_{X_{\alpha}}(1 - \varepsilon) \leq L$ for all $\alpha$. The first
inequality yields
\[
- L < \iqf_{X_{\alpha}}(\varepsilon) - \iqf_{X_{\alpha
}}(\varepsilon/2) =
\int_{\varepsilon/2}^{\varepsilon}q_{X_{\alpha}}^{L}(s)\,ds
\leq\frac{\varepsilon}{2} q_{X_{\alpha}}^{L}(\varepsilon) \text{,}
\]
which shows that $- \tfrac{2L}{\varepsilon} < q_{X_{\alpha
}}^{L}(\varepsilon)$ and, hence, $F_{X_{\alpha}} (- \tfrac
{2L}{\varepsilon} ) < \varepsilon$. Similarly, from the second
inequality, one gets $F_{X_{\alpha}} (\tfrac{2L}{\varepsilon} )
\geq1 - \varepsilon$ for all $\alpha$. This proves the tightness of
the laws of $X_{\alpha}$.

Since implications (iv)~$\Rightarrow$~(ii) and (iii)~$\Rightarrow$~(ii)
are obvious, the claim follows.
\end{proof}

\begin{thm}\label{htsr6682hsgzq2}
Let $\{X_{\alpha}\}$ be a family of random variables. Then the
following statements are equivalent\/{\rm:}
\begin{itemize}
\item[\emph{(i)}] The family $\{X_{\alpha}\}$ is uniformly integrable.
\item[\emph{(ii)}] The family of integrated quantile functions $\{
\iqf
_{X_{\alpha}}\}$ is equicontinuous on $[0,\,1]$.
\item[\emph{(iii)}] The family of integrated quantile functions $\{
\iqf
_{X_{\alpha}}\}$ is relatively compact in the space $C[0,\,1]$ of
continuous functions with supremum norm.\vadjust{\goodbreak}
\end{itemize}
\end{thm}

\begin{proof}
Let us consider the probability space $(\varOmega,\, \cF,\, \PP)$ as in
the proof of Theorem~\ref{hs6gsvgaxs3} and define random variables
$Y_{\alpha}(\omega) = q_{X_{\alpha}}^{L}(\omega)$. Then $X_{\alpha}
\overset{d}{=} Y_{\alpha}$ and it is enough to study the uniform
integrability of the family $\{Y_{\alpha}\}$. Without loss of
generality, we suppose that $X_{\alpha} = Y_{\alpha}$.

Let us recall that a family $\{X_{\alpha}\}$ is uniformly integrable if
and only if $\EE|X_{\alpha}|$ are bounded and $\EE[|X_{\alpha
}|\mathbb
{1}_{A}]$ are uniformly continuous, i.\,e. $\sup_{\alpha}\EE
[|X_{\alpha
}|\mathbb{1}_{A}] \rightarrow0$ as $\PP(A) \rightarrow0$; moreover,
the boundedness of $\EE|X_{\alpha}|$ is a consequence of the uniform
continuity if the measure $\PP$ has no atomic part, in particular, in
our case. On the other hand, by the Arzela--Ascoli theorem a set in
$C[0,\,1]$ is relatively compact if and only if it is uniformly bounded
and equicontinuous.

We shall check that the uniform boundedness and the equicontinuity of
$\{\iqf_{X_{\alpha}}\}$ are equivalent to uniform boundedness of $\EE
|X_{\alpha}|$ and the uniform continuity of $\EE[|X_{\alpha}|\mathbb
{1}_{A}]$, respectively. In view of the above this is sufficient for
the proof of the theorem.

By the properties of integrated quantile functions,
\[
\sup_{u \in[0,\,1]}\iqf_{X_{\alpha}}(u) = \max\bigl(\EE
\bigl[X_{\alpha}^{-}\bigr], \, \EE\bigl[X_{\alpha}^{+}
\bigr]\bigr) \leq\EE|X_{\alpha}| \leq2\sup_{u \in[0,\,1]}\iqf
_{X_{\alpha}}(u) \text{.}
\]
Hence, $\sup_{\alpha}\EE|X_{\alpha}| < \infty$ if and only if the
family $\{\iqf_{X_{\alpha}}\}$ is uniformly bounded.

For a fixed $\varepsilon> 0$, let $\delta> 0$ be such that $\sup
_{\alpha}\EE[|X_{\alpha}|\mathbb{1}_{A}] < \varepsilon$ for any Borel
set $A \subseteq(0,\,1)$ with $\PP(A) < \delta$. Let $u_1, u_2 \in
[0,\,1]$ satisfy $0 < u_2 - u_1 < \delta$. Then, for any $\alpha$,
\[
\big|\iqf_{X_{\alpha}}(u_2) - \iqf_{X_{\alpha}}(u_1)\big| =
\Bigg|\int_{u_1}^{u_2}X_{\alpha}(\omega)\,d\omega\Bigg|
\leq\int_{u_1}^{u_2}\big|X_{\alpha}(\omega)\big|\,d
\omega< \varepsilon\text{.}
\]

Conversely, fix $\varepsilon> 0$ and let $\delta> 0$ be such that
$|\iqf_{X_{\alpha}}(u_2) - \iqf_{X_{\alpha}}(u_1)| < \varepsilon$ for
all $\alpha$ if $|u_2 - u_1| < \delta$. Since $X_{\alpha}(\omega)$ is
increasing in $\omega$, the following inequality holds for any Borel
subset $A \subseteq(0,\,1)$:
\begin{equation}
\label{ianys2837kjs3d} \int_{(0, \, \PP(A)]}X_{\alpha}(\omega)\,
d\omega\leq
\int_{A}X_{\alpha
}(\omega)\,d\omega\leq\int
_{[1-\PP(A), \, 1)}X_{\alpha}(\omega)\, d\omega\text{.}
\end{equation}
Therefore, if $\PP(A) < \delta$ then
\begin{align*}
\EE\bigl[|X_{\alpha}|\mathbb{1}_{A}\bigr] &= \int_{A \cap\{
X_{\alpha} < 0\}
}-X_{\alpha}(
\omega)\,d\omega+ \int_{A \cap\{X_{\alpha} > 0\}
}X_{\alpha}(\omega)\,d\omega
\\
&\leq\int_{(0,\,\PP(A \cap\{X_{\alpha} < 0\})]}-X_{\alpha}(\omega
)\, d\omega+ \int
_{[1 - \PP(A \cap\{X_{\alpha} > 0\}),\,1)}X_{\alpha
}(\omega)\,d\omega
\\
&\leq\max_{u \in(0,\PP(A)]} \bigl(\iqf_{X_{\alpha}}(0) - \iqf
_{X_{\alpha}}(u) \bigr) + \max_{u \in[1 - \PP(A),1)} \bigl(\iqf
_{X_{\alpha}}(1) - \iqf_{X_{\alpha}}(u) \bigr)
\\
& < 2 \varepsilon\text{.}\qedhere
\end{align*}
\end{proof}

The following criterion of uniform integrability is proved in \cite
{LeskelaVihola2013}.

\begin{thm}[Leskel\"a and Vihola]
A family $\{X_{\alpha}\}$ of integrable random variables is uniformly
integrable if and only if there is an integrable random variable $X$
such that $|X_{\alpha}| \leq_{icx} X$ for all $\alpha$.
\end{thm}

\begin{proof}
Without loss of generality, we may assume that all $X_{\alpha}$ are
nonnegative. To simplify notation, let $K_{\alpha}(u) := \iqf
_{X_{\alpha
}}^{[1]}(u)$, $u \in[0,\,1]$. Then $K_{\alpha}$ are increasing
continuous convex functions with $K_{\alpha}(1) = 0$. According to
Theorems~\ref{hs6gsvgaxs3}, \ref{jhdt67thdbba}~(ii) and \ref
{htsr6682hsgzq2}, it is enough to prove that the family $\{K_{\alpha}\}
$ is equicontinuous if and only if there is an increasing continuous
convex function $K(u)$, $u \in[0,\,1]$, with $K(1) = 0$ such that
\[
K_{\alpha}(u) \geq K(u)\quad\text{for all}\ u \in[0,\,1]\ \text
{for all}\ \alpha.
\]

The sufficiency is evident. Indeed, if $0 \leq u_1 \leq u_2 \leq1$,
then, for all $\alpha$,
\begin{align*}
0 \leq K_{\alpha}(u_2) - K_{\alpha}(u_1) &
\leq K_{\alpha}(1) - K_{\alpha}\bigl(1 - (u_2-u_1)
\bigr) = - K_{\alpha}\bigl(1 - (u_2-u_1)\bigr)
\\
& \leq-K\bigl(1 - (u_2-u_1)\bigr) = K(1) - K\bigl(1 -
(u_2-u_1)\bigr) \text{,}
\end{align*}
and the equicontinuity follows from the continuity of $K$.

Let us define $K$ as the lower semicontinuous convex envelope of $\inf
_{\alpha}K_{\alpha}$.
To prove the necessity, it is enough to show that
$K(1) = 0$ if the family $\{K_{\alpha}\}$ is equicontinuous. Fix
$\varepsilon> 0$ and let $\delta> 0$ be such that $|K_{\alpha}(u_2) -
K_{\alpha}(u_1)| < \varepsilon$ for all $\alpha$ if $|u_2 - u_1|
\leq
\delta$. In particular, $K_{\alpha}(1 - \delta) > - \varepsilon$. Since
$K_{\alpha}$ is convex, we have $K_{\alpha}(u) > - \tfrac
{\varepsilon
}{\delta} (1 - u)$ for all $u \in[0, \; 1 - \delta]$ and for all
$\alpha$. Moreover, since $K_{\alpha}$ is increasing, $K_{\alpha}(u)
\geq- \varepsilon$ for all $u \in[1 - \delta, \; 1]$ and for all
$\alpha$. Combining, we get
\[
\inf_{\alpha}K_{\alpha}(u) \geq\min\biggl(-
\frac{\varepsilon}{\delta
} (1 - u), - \varepsilon\biggr) \geq- \frac{\varepsilon}{\delta} +
\frac{\varepsilon(1-\delta)}{\delta} u \text{,}
\]
for all $u \in[0,\,1]$. It follows that $K(u) \geq- \tfrac
{\varepsilon
}{\delta} + \tfrac{\varepsilon(1-\delta)}{\delta} u$ for all $u
\in
[0,\,1]$, in particular, $K(1-\delta) > -2\varepsilon$. The claim follows.
\end{proof}

\subsection{Weak convergence}

In this subsection $(X_n)$ is a sequence of random variables.

\begin{thm}\label{rte5rgvsaq}
The following statements are equivalent\/{\rm:}
\begin{itemize}
\item[\emph{(i)}] The sequence $(X_n)$ weakly converges.
\item[\emph{(ii)}] There is a sequence $(c_n)$ of numbers such that,
for every $u \in(0,\,1)$, the sequence $(\iqf_{X_n}(u)-c_n )$
converges to a finite limit.
\item[\emph{(iii)}] The sequence $(\iqf_{X_n})$ converges uniformly on
every $[\alpha, \, \beta] \subseteq(0,\,1)$.
\end{itemize}

Moreover{\rm,} in this case if $X$ is a weak limit of $(X_n)$, then
$\iqf_{X}(u) = \lim_{n \rightarrow\infty}\iqf_{X_n}(u)$ for all $u
\in
(0,\,1)$.
\end{thm}

\begin{remark}\label{dju63jhs}
If $\EE[X_n^{-}] < \infty$ $($resp. $\EE[X_n^{+}] < \infty)$ for all
$n$, then the pointwise convergence of $\iqf_{X_n}^{[0]}$ $($resp.
$\iqf
_{X_n}^{[1]})$ on $(0,\,1)$ is sufficient $($use Theorem~\ref
{rte5rgvsaq}, $\mathrm{(ii)}$ $\Rightarrow$ $\mathrm{(i)})$ but not
necessary for the weak convergence of $X_n$.
\end{remark}

\begin{thm}\label{jny6hgbhsjtr}
Let $(X_n)$ weakly converge and $\EE|X_n| < \infty$ {\rm(}resp. $\EE
[X_n^{-}] < \infty$, resp. $\EE[X_n^{+}] < \infty)$. Then the following
statements are equivalent\/{\rm:}
\begin{itemize}
\item[\emph{(i)}] The sequence $(|X_n|)$ {\rm(}resp. $(X_n^{-})$,
resp. $(X_n^{+}))$ is uniformly integrable.
\item[\emph{(ii)}] The sequence of functions $(\iqf_{X_n})$ converges
pointwise on $[0,\,1]$ {\rm(}resp. $[0,\,1)$, resp. $(0,\,1])$ to a
continuous function with finite values.
\item[\emph{(iii)}] The sequence $(\iqf_{X_n})$ converges uniformly to
a finite-valued function on $[0, \, 1]$ {\rm(}resp. on every $[0,\,
\beta] \subseteq[0,\,1)$, resp. on every $[\alpha,\,1] \subseteq
(0,\,1])$.
\end{itemize}
\end{thm}

\begin{remark}\label{dsfse3sa}
In contrast to Remark~\ref{dju63jhs}, a combination of the weak
convergence of $X_n$ and the uniform integrability of $X_n^{-}$
$($resp. $X_n^{+})$ can be expressed in terms of the shifted integrated
quantile functions $\iqf_{X_n}^{[0]}$ $($resp. $\iqf_{X_n}^{[1]})$. For
instance, let a sequence $(X_n)$ weakly converge to $X$ and the
sequence $(X_n^{+})$ is uniformly integrable. Then the pointwise limit
of $\iqf_{X_n}^{[1]}(u)$ satisfies
\begin{align}
\label{jhdf3hrhd7} \lim_{n \rightarrow\infty}\iqf_{X_n}^{[1]}(u)
&= \lim_{n \rightarrow
\infty}\iqf_{X_n}(u) - \lim_{n \rightarrow\infty}
\EE\bigl[X_n^{+}\bigr]
\nonumber
\\
&= \iqf_{X}(u) - \EE\bigl[X^{+}\bigr] =
\iqf_{X}^{[1]}(u) , \quad u \in(0,\,1],
\end{align}
and is continuous on $(0,\,1]$. Conversely, if the functions $\iqf
_{X_n}^{[1]}$ converge pointwise to a continuous limit on $(0,\,1]$,
then $X_n$ weakly converges, say, to $X$ $($use Theorem~\ref
{rte5rgvsaq}, $\mathrm{(ii)}$ $\Rightarrow$ $\mathrm{(i)})$. In
particular, for any $u \in(0,\,1)$,
\[
\lim_{n \rightarrow\infty}\iqf_{X_n}^{[1]}(u) = \lim
_{n \rightarrow
\infty}\iqf_{X_n}(u) - \lim_{n \rightarrow\infty}\EE
\bigl[X_n^{+}\bigr] = \iqf_{X}(u) - \lim
_{n \rightarrow\infty}\EE\bigl[X_n^{+}\bigr] \text{.}
\]
Continuity of the limiting function in the left-hand side of the above
formula at $u = 1$ implies $\lim_{n \rightarrow\infty}\EE[X_n^{+}] =
\lim_{u \uparrow1}\iqf_{X}(u) = \EE[X^{+}]$.
\end{remark}

\begin{proof}[Proof of Theorems~\ref{rte5rgvsaq} and \ref{jny6hgbhsjtr}]
First, let us suppose that $(X_n)$ weakly converges to $X$. It is well
known that then $q_{X_n}^{L}(u) \rightarrow q_{X}^{L}(u)$ as $n
\rightarrow\infty$ for every continuity point $u$ of $q_X^{L}$. Put
$u_{n,0} := F_{X_n}(0)$ and $u_{0} := F_{X}(0)$.

Assume for the moment that $X_n$ are uniformly bounded. Then, for any
$u \in[0,\,1]$,
\begin{align*}
\iqf_{X_n}(u) &= \int_{0}^{u}q_{X_n}(s)
\,ds - \int_{0}^{u_{n,0}}q_{X_n}(s)\,ds = \int
_{0}^{u}q_{X_n}(s)\,ds + \int
_{0}^{1}\bigl(q_{X_n}(s)
\bigr)^{-}\,ds
\\
&\rightarrow\int_{0}^{u}q_{X}(s)\,ds +
\int_{0}^{1}\bigl(q_{X}(s)
\bigr)^{-}\,ds = \iqf_X(u)
\end{align*}
by the dominated convergence theorem. Moreover, by Theorem~\ref
{htsr6682hsgzq2} the sequence $(\iqf_{X_n})$ is relatively compact in
$C[0,\,1]$. Combined with pointwise convergence, this shows that $(\iqf
_{X_n})$ converges to $\iqf_X$ uniformly on $[0,\,1]$.

If no assumptions on $X_n$ are imposed, let us introduce the function
$g_C(x) := \max(\min(x, \,C), \,-C)$, $C > 0$, and define random variables
\[
Y_n := g_C(X_n) , \qquad Y :=
g_C(X) \text{.}
\]
Then $(Y_n)$ weakly converges to $Y$. Hence, $\iqf_{Y_n} \rightarrow
\iqf_Y$ uniformly on $[0,\,1]$ as it has just been proved. However,
$\iqf_{Y_n} = \iqf_{X_n}$ on $[F_{X_n}(-C),\, F_{X_n}(C)] \ni u_{n,0}$
and $\iqf_{Y} = \iqf_{X}$ on $[F_{X}(-C),\, F_{X}(C)] \ni u_{0}$. Given
$[\alpha, \, \beta] \subseteq(0,\,1)$, choose $C > 0$ so that
$[\alpha
, \, \beta] \subseteq[F_{X}(-C),\, F_{X}(C)]$ and $[\alpha, \, \beta]
\subseteq[F_{X_n}(-C),\, F_{X_n}(C)]$ for all $n$, which is possible
by tightness. Therefore, $\iqf_{X_n}(u) \rightarrow\iqf_X(u)$
uniformly in $u$ on $[\alpha, \, \beta] \subseteq(0,\,1)$. In
particular, $(\iqf_{X_n})$ converges pointwise to $\iqf_X$ on $(0,\,1)$.

To complete the proof of Theorem~\ref{rte5rgvsaq} it remains to prove
implication (ii) $\Rightarrow$ (i). Let $u,v\in(0,\,1)$. By the
assumption, the sequence $\iqf_{X_n}(u) - \iqf_{X_n}(v)$ converges to a
finite limit and, hence, is bounded. By Theorem \ref{t:tightness}, the
laws of $X_n$ are tight. Let $(X_{n_k})$ be a weakly convergent
subsequence. It follows from what has been proved that the integrated
quantile function $K(u)$ of its limit coincides with $\lim_{k
\rightarrow\infty}\iqf_{X_{n_k}}(u)$ for $u \in(0,\,1)$. Therefore,
for all $u \in(0,\,1)$,
\[
\lim_{n \rightarrow\infty} \bigl(\iqf_{X_n}(u)-c_n \bigr)
= \lim_{k \rightarrow
\infty} \bigl(\iqf_{X_{n_k}}(u)-c_{n_k}
\bigr) = K(u) - \lim_{k \rightarrow\infty
} c_{n_k} \text{.}
\]
This implies that $c_{n_k}$ converges to a finite limit and that $K(u)$
is obtained from $\lim_{n \rightarrow\infty} (\iqf_{X_n}(u)-c_n )$ by
adding a constant. Since $K$ is an integrated quantile function, this
constant is determined uniquely. Thus, $K$ is the same for all weakly
convergent subsequences, which means that $(X_n)$ weakly converges.

It is enough to prove Theorem \ref{jny6hgbhsjtr} in one of three cases,
for example, in the case $\EE[X_n^{-}] < \infty$. Assume that
$(X_n^{-})$ is uniformly integrable. Then $\EE[X_n^{-}] \rightarrow
\EE
[X^{-}]$, where $X$ is a weak limit of $(X_n)$. In other words, $\iqf
_{X_n}(0) \rightarrow\iqf_X(0)$. Thus, we have (ii). Moreover, the
sequence $(\iqf_{X_n^{-}})$ is equicontinuous. It follows that $(\iqf
_{X_n})$ converges uniformly on every segment $[0,\,\beta] \subseteq
(0,\,1)$. Implication (iii) $\Rightarrow$ (ii) is trivial. If (ii)
holds, then $\lim_{n \rightarrow\infty}\iqf_{X_n}(u)$ is a continuous
function in $u \in[0,\,1)$. On the other hand, this limit is $\iqf
_X(u)$ for $u \in(0,\,1)$. Hence, $\EE[X^{-}] = \iqf_X(0) = \lim_{n
\rightarrow\infty}\EE[X_n^{-}]$, and the sequence $(X_n^{-})$ is
uniformly integrable.
\end{proof}

\section{Applications to binary statistical models}\label{s:sec4}

The theory of statistical experiments deals with the problem of
comparing the information in different experiments. The foundation of
the theory of experiments was laid by Blackwell \cite
{Blackwell1951,Blackwell1953},
who first studied a notion of being more informative
for experiments. Since it is difficult to give an explicit definition
of statistical information, the theory of statistical experiments
evaluates the performance of an experiment in terms of the set of
available risk functions, in general, for arbitrary decision spaces and
loss functions. For the theory of statistical experiments we refer to
\cite{Lecam1986,ShiryaevSpokoyniy2000}, and especially to \cite
{Strasser1985,Torgersen1991}, where the reader can find unexplained
results and additional information.

In this paper we consider only binary statistical experiments, or
dichotomies, $\E=(\varOmega,\cF,\PP,\PP')$. It is known that for binary
models, it is enough to deal with testing problems, i.\,e. with tests
as decision rules and with the probabilities of errors of the first and
the second kinds of a test.

Let us introduce some notation. $\QQ$ is any probability measure
dominating $\PP$ and $\PP'$, $z := d\PP/ d \QQ$ and $z' := d \PP'
/ d
\QQ$ are the corresponding Radon--Nikod\'ym derivatives. $\EE$, $\EE'$,
and $\EE_{\QQ}$ are the expectations with respect to $\PP$, $\PP'$ and
$\QQ$ respectively. Note that $\PP(z = 0) = 0$ and $Z := z' / z$, where
$0 / 0 = 0$ by convention, is the Radon--Nikod\'ym derivative of the
$\PP$-absolutely continuous part of $\PP'$ with respect to $\PP$.\looseness=-1

For an experiment $\E=(\varOmega,\cF,\PP,\PP')$, denote by
$\varPhi(\E)$ the
set of all test functions $\varphi$ in $\E$, i.e.\ measurable mappings
from $(\varOmega,\cF)$ to $[0,1]$. It is convenient for us to interpret
$\varphi(\omega)$ as the probability to accept the null hypothesis
$\PP
$ and to reject the alternative $\PP'$ if $\omega$ is observed. Then
$\alpha(\varphi):=\EE[1-\varphi]$ and $\beta(\varphi):=\EE'
[\varphi]$
are the probabilities of errors of the first and the second kind
respectively of a test $\varphi$.

Denote
\[
\gN(\E):= \bigl\{ \bigl(\EE[\varphi],\,\EE'[\varphi] \bigr
)\colon
\varphi\in\varPhi(\E) \bigr\} = \bigl\{ \bigl(1 - \alpha
(\varphi),\,\beta(
\varphi) \bigr)\colon\varphi\in\varPhi(\E) \bigr\}.
\]
It is well known that $\gN(\E)$ is a convex and closed subset of
$[0,1]\times[0,1]$, contains $(0,0)$, and is symmetric with respect to
the point $(1/2,1/2)$, see, e.g., \cite[p.~62]{LehmannRomano2005}. In
Fig.~\ref{f_j6fghf} we present a set $\gN(\E)$ of generic form.
Introduce also the \textit{risk function}
\begin{equation}
\label{e:KE} \rr_\E(u):= \inf\, \bigl\{\beta(\varphi)\colon
\varphi
\in\varPhi(\E),\ \alpha(\varphi)=u \bigr\} , \quad u \in[0,\,1],
\end{equation}
that is the smallest probability of the second kind error if the
probability of the first kind error is $u$.
It follows that the set $\gN(\E)$ and the risk function $\rr_\E$ are
connected by
\begin{equation}
\label{rviaN} \rr_\E(u) = \inf\,\bigl\{v \colon(1 - u, v) \in\gN(
\E)\bigr\}
\end{equation}
and
\begin{equation}
\label{Nviar} \gN(\E) = \bigl\{(u,\,v) \in[0,\,1] \times[0,\,1]
\colon
\rr_\E(1-u) \leq v \leq1 - \rr_\E(u) \bigr\} \text{.}
\end{equation}
In particular, $\rr_{\E}$ is a continuous convex decreasing function
taking values in $[0,\,1]$ and $\rr_{\E}(1) = 0$. Therefore, by
Theorem~\ref{hs6gsvgaxs3}, $\rr_{\E}(u)$ coincides on $[0,\,1]$ with an
integrated quantile function corresponding to some distribution. The
following result determines this distribution and explains why it is
natural to use integrated quantile functions for binary models.

\begin{figure}[t]
\includegraphics{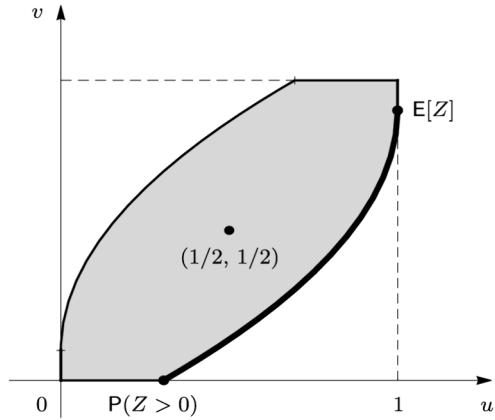}
\caption{The shaded area represents the set $\gN(\E)$. The thick curve
corresponds to admissible, or Neyman--Pearson tests $\varphi^*$ with
the following property: if $\alpha(\varphi)\leq\alpha(\varphi^*)$ and
$\beta(\varphi)\leq\beta(\varphi^*)$ for some $\varphi\in\varPhi
(\E)$,
then both inequalities are equalities. See \cite
[Chapter~2]{Strasser1985} for more details. The thick curve together
with the horizontal segment $[0,\,\PP(Z>0)]\times\{0\}$ is the graph of
the function $\rr_\E(1-u)$}
\label{f_j6fghf}
\end{figure}

\begin{prop}\label{jydy3gh4hsj}
For all $u \in[0,\,1]$, $\rr_\E(u) = \iqf_{-Z}(u)$, where $\iqf_{-Z}$
is the integrated quantile function corresponding to the distribution
of the negative likelihood ratio $-Z = - z' / z$ under the null hypothesis.
\end{prop}

\begin{proof}
Let $\varphi_0\in\varPhi(\E)$ and $x\in\bbR_+$. Then the straight line
with the slope $x$ and passing through the point $ (\EE[\varphi
_0],\,\EE'[\varphi_0] )$ lies below the graph of $K(u) := \rr_{\E
}(1-u)$ on $[0,\,1]$ if and only if, for every $\varphi\in\varPhi(\E)$,
\[
\EE'[\varphi] \geq\EE'[\varphi_0] + x
\bigl(\EE[\varphi] - \EE[\varphi_0]\bigr).
\]
Passing to a dominating measure $\QQ$, the above inequality can be
rewritten as
\[
\EE_\QQ\bigl[\bigl(z'-xz\bigr) (\varphi-
\varphi_0)\bigr]\geq0.
\]
This holds for every $\varphi\in\varPhi(\E)$ if and only if
\begin{equation}
\label{e:ns} \varphi_0 = \mathbb{1}_{\{z'<xz\}} +
\varphi_0 \mathbb{1}_{\{z'=xz\}} \quad\text{$\QQ$-a.s.}
\end{equation}

Let $u \in(0,\,1)$ and take any $x\in[q_Z^L(u),q_Z^R(u)]$. Then $u\in
[F_Z(x-0),F_Z(x)]$, so there is $\gamma\in[0,\,1]$ such that $u =
(1-\gamma)F_Z(x-0)+\gamma F_Z(x)$. Finally, put $\varphi_0:=\mathbb
{1}_{\{z'<xz\}} + \gamma\mathbb{1}_{\{z'=xz\}}$. Since $Z=z'/z$ $\PP
$-a.s., we get $\EE[\varphi_0]=u$ and, obviously, $\varphi_0$ satisfies
\eqref{e:ns}.
This means that $x\in\partial K(u)$. Conversely, let $u\in(0,\,1)$ and
$x\in\partial K(u)$. Take any $\varphi_0\in\varPhi(\E) $ such that
$\EE
[\varphi_0]=u$ and $\beta(\varphi_0)=K(u)$. Then $\varphi_0$ satisfies
\eqref{e:ns}, which implies $\EE[\varphi_0]\in[F_Z(x-0),F_Z(x)]$.
Hence, $x\in[q_Z^L(u),q_Z^R(u)]$.

It is clear that $K(0)=\iqf_Z(0)=0$. Now taking into account that $K$
and $\iqf_Z$ are convex functions, $\iqf_Z$ is continuous on $[0,\,1]$,
and $\partial K(u)=\partial\iqf_Z(u)$ for $u\in(0,\,1)$, it remains to
prove that $K(1)\leq\iqf_Z(1)=\EE[Z]$. This is easy: take $\varphi
_0 :=
\mathbb{1}_{\{z>0\}}$, then $\EE[\varphi_0]=1$ and
$\EE'[\varphi_0]=\PP' (z>0) = \EE_\QQ[z'\mathbb{1}_{\{z>0\}}] =
\EE[Z]$.
Finally, $\rr_{\E}(u) = K(1- u) = \iqf_{Z}(1 - u) = \iqf_{-Z}(u)$.
\end{proof}

\begin{remark} A usual way to prove that the set $\gN(\E)$ is closed is
based on weak compactness of test functions, see, e.g., \cite
{LehmannRomano2005}. The reader may readily verify that the closedness
of $\gN(\E)$ follows directly from the above proof.
\end{remark}

Let us also introduce the \textit{minimum Bayes risk function} (the
\textit{error function})
\[
\bb_{\E}(\pi) := \inf_{\varphi\in\varPhi(\E)} \bigl((1 - \pi
)\alpha
(\varphi) + \pi\beta(\varphi) \bigr) ,\quad\pi\in[0,\,1].
\]
It can be expressed in terms of risk function $\rr_{\E}$ and vice
versa. Indeed, for any $\pi\in(0,\,1)$,
\begin{align}
\label{e:J2B} \bb_{\E}(\pi) &= \inf_{u \in[0,\,1]} \bigl((1
- \pi)u + \pi\rr_{\E
}(u) \bigr)
\nonumber
\\
&= - \pi\sup_{u \in[0,\,1]} \biggl(-\frac{1-\pi}{\pi}u -
\rr_{\E
}(u) \biggr) = - \pi\sup_{u \in[0,\,1]} \biggl(-
\frac{1-\pi}{\pi}u - \iqf_{-Z}(u) \biggr) =
\nonumber
\\
&= - \pi\idf_{-Z}\biggl(-\frac{1-\pi}{\pi}\biggr) = 1 - \pi-
\pi
\idf_{Z}\biggl(\frac
{1-\pi}{\pi}\biggr) \text{.}
\end{align}
In particular, it follows from Theorem~\ref{ystwqe52dgh} that
\begin{align*}
\lim_{\pi\downarrow0} \frac{\bb_{\E}(\pi)}{\pi} &= \lim_{x\to
+\infty
}
\bigl(x-\idf_Z(x)\bigr) = \EE[Z],
\\
\lim_{\pi\uparrow1} \frac{\bb_{\E}(\pi)}{1-\pi} &= 1-\lim
_{x\downarrow
0}
\frac{\idf_Z(x)}{x} = \PP(Z>0),
\end{align*}
see \cite[Lemma 14.6]{Strasser1985} and \cite
[p.~607]{Torgersen1991}.\vadjust{\goodbreak}

Conversely, using Definition~\ref{d:iqf} and \eqref{e:J2B}, we get, for
$u\in[0,\,1]$,
\begin{align}
\label{e:B2J} \rr_{\E}(u) &= \iqf_{-Z}(u) = \sup
_{x \in\mathbb{R}} \bigl(xu - \idf_{-Z}(x) \bigr) = \sup
_{x<0} \bigl(xu - \idf_{-Z}(x) \bigr)
\nonumber
\\
&= \sup_{\pi\in(0,\,1)} \biggl(-\frac{1-\pi}{\pi}u -
\idf_{-Z} \biggl(-\frac{1-\pi}{\pi} \biggr) \biggr) =\sup
_{\pi\in(0,\,1)}\frac{1}{\pi} \bigl( \bb_{\E}(\pi) - (1 -
\pi)u \bigr) \text{},
\end{align}
see \cite[p.~590]{Torgersen1991}. Here we have used that $\idf
_{-Z}(x)=x$ for $x\geq0$.

Finally, let us introduce one more characteristic of binary models,
namely the distribution of the `likelihood ratio'
\[
\mu_{\E}(A) := \PP(Z \in A) , \quad A \in\mathscr{B}(
\mathbb{R}_{+}).
\]

Now let us present some basic notions and results from the theory of
comparison of dichotomies. All these facts are well known, see e.\,g.
\cite[Chapter~3]{Strasser1985} and \cite[Chapter~10]{Torgersen1991}.
Our aim is to show how they can be deduced with the help of the results
in Sections~\ref{s:sec2} and~\ref{s:sec3}.

\begin{defin}
Let $\E=(\varOmega,\cF,\PP,\PP')$ and $\widetilde{\E
}=(\widetilde{\varOmega
},\widetilde{\cF},\widetilde{\PP},\widetilde{\PP}')$ be two binary
experiments. $\E$ is said to be \textit{more informative} than
$\widetilde{\E}$, denoted by $\E\succeq\widetilde{\E}$ or
$\widetilde
{\E} \preceq\E$, if $\gN(\E) \supseteq\gN(\widetilde{\E})$.
$\E$ and
$\widetilde{\E}$ are called \textit{equivalent} ($\E\sim\widetilde
{\E
}$) if $\E\succeq\widetilde{\E}$ and $\E\preceq\widetilde{\E}$. The
\textit{type} of an experiment is the totality of all experiments which
are equivalent to the given experiment.
\end{defin}

\begin{prop}
Let $\E$ and $\widetilde{\E}$ be binary experiments. The following
statements are equivalent\/{\rm:}
\begin{itemize}
\item[\emph{(i)}] $\E\succeq\widetilde{\E}$;
\item[\emph{(ii)}] $\rr_\E\leq\rr_{\widetilde{\E}}$;
\item[\emph{(iii)}] $\bb_\E\leq\bb_{\widetilde{\E}}$;
\item[\emph{(iv)}] $\mu_{\widetilde{\E}} \leq_{decx} \mu_\E$.
\end{itemize}
\end{prop}

\begin{proof}
(i)$\;\Leftrightarrow\;$(ii) follows from \eqref{rviaN} and \eqref
{Nviar}. (ii)$\;\Leftrightarrow\;$(iii) is a consequence of \eqref
{e:J2B} and \eqref{e:B2J}. Finally, (ii)$\;\Leftrightarrow\;$(iv)
follows from Proposition~\ref{jydy3gh4hsj} and Theorem~\ref{jhdt67thdbba}.
\end{proof}

\begin{cor}
Let $\E$ and $\widetilde{\E}$ be binary experiments. The following
statements are equivalent\/{\rm:}
\begin{itemize}
\item[\emph{(i)}] $\E\sim\widetilde{\E}$;
\item[\emph{(ii)}] $\rr_\E=\rr_{\widetilde{\E}}$;
\item[\emph{(iii)}] $\bb_\E=\bb_{\widetilde{\E}}$;
\item[\emph{(iv)}] $\mu_\E=\mu_{\widetilde{\E}}$.
\end{itemize}
\end{cor}

\begin{prop}
\begin{itemize}
\item[\emph{(i)}] The mapping $\E\rightsquigarrow\rr_\E$ is onto the
set of all convex continuous decreasing functions $\rr\colon[0,\,1]
\to[0,\,1]$ such that $\rr(1) = 0$.
\item[\emph{(ii)}] The mapping $\E\rightsquigarrow\bb_\E$ is onto
the set of all concave functions $\bb\colon[0,\,1] \to[0,\,1]$ such
that $\bb(\pi) \leq\pi\wedge(1-\pi)$.
\item[\emph{(iii)}] The mapping $\E\rightsquigarrow\mu_{\E}$ is onto
the set of all probability measures $\mu$ on\break$(\mathbb{R}_{+},
\,
\mathscr{B}(\mathbb{R}_{+}))$ such that $\int x \, \mu(dx) \leq1$.
\end{itemize}
\end{prop}

\begin{proof} (i) Let $\rr(u)$, $u \in[0, \,1]$, be a convex
continuous decreasing function with $\rr(1) = 0$ and $\rr(0) \leq1$.
Let $\varOmega= [0,\,1]$ and $\cF$ be the Borel $\sigma$-field. Define
$\PP$ on $(\varOmega, \, \cF)$ as the Lebesgue measure and $\PP'$
as the
measure with the distribution function
\begin{equation}
\label{e:repr1} F(x) = \lleft\{ %
\begin{array}{@{}ll}
0, & \text{if $x < 0$,} \\
\rr(1-x), & \text{if $0 \leq x < 1$,} \\
1, & \text{if $x \geq1$.}
\end{array}
\rright.
\end{equation}
Then $Z(u) = F'_{-}(u)$ $\PP$-a.s. As in the proof of Theorem~\ref
{hs6gsvgaxs3}, it follows that $\rr(1-u) = \iqf_{Z}(u)$.
Proposition~\ref{jydy3gh4hsj} allows us to conclude that $\rr_{\E} =
\rr$.

(iii) First, it is evident that $\mu_{\E}$ is a probability measure on
$(\mathbb{R}_{+}, \, \mathscr{B}(\mathbb{R}_{+}))$ such that $\int x
\,
\mu_{\E}(dx) \leq1$ for any dichotomy $\E$. Now, let $\mu$ be a
probability measure on $(\mathbb{R}_{+}, \, \mathscr{B}(\mathbb
{R}_{+}))$ such that $\int x \, \mu(dx) \leq1$. Put $\varOmega= [0, +
\infty]$ and let $\cF$ be the Borel $\sigma$-field. Define $\PP$ as the
probability measure which coincides with $\mu$ on Borel subsets of
$\mathbb{R}_{+}$. Finally, define $\PP'$ by
\[
\PP'(B \cap\mathbb{R}_{+}) := \int_{B \cap\mathbb{R}_{+}}x
\,\mu(dx) , \qquad\PP'\bigl(\{+\infty\}\bigr) := 1 - \int
_{\mathbb{R}_{+}}x\,\mu(dx) \text{.}
\]
If $\E$ is defined as $\E= (\varOmega, \, \cF, \, \PP, \, \PP')$,
it is
clear that $\mu_{\E} = \mu$.

(ii) First, it follows from the definition of the error function that
$0 \leq\bb_{\E}(\pi) \leq\pi\wedge(1 - \pi)$, $\pi\in[0,\,1]$,
and that $\bb_{\E}$ is concave. If $\bb$ is a function with these
properties, then define $\idf(x) := x - (1 + x) \bb(\tfrac{1}{1 + x})$,
$x \geq0$, cf. \eqref{e:J2B}; put also $\idf(x) = 0$ for $x < 0$.
Using concavity of $\bb$, it is easy to check that $\idf$ is convex on
$\mathbb{R}_{+}$. Since $\bb(0) = 0$, we have $\lim_{x \rightarrow+
\infty}\tfrac{\idf(x)}{x} = 1$. The inequalities $0 \leq\bb(\pi)
\leq
1 - \pi$ imply that $0 \leq\idf(x) \leq x$ for all $x \geq0$. In
particular, $\idf$ is convex on $\mathbb{R}$ and, by Theorem~\ref
{v6gsh67htyfad}, $\idf$ is the integrated distribution function of some
nonnegative random variable $Z$. Finally, the inequality $\bb(\pi)
\leq
\pi$ implies that $\idf(x) \geq x - 1$, which means that $\EE[Z]
\leq
1$ by Theorem~\ref{ystwqe52dgh}. Hence, $\bb$ is the error function of
an experiment $\E$ such that $\mu_{\E} = \operatorname{Law}(Z)$.
\end{proof}

Let us note that the proofs of (i) and (iii) give more than it is
stated. Starting with a function $\rr$ or a measure $\mu$ from
corresponding classes, we construct an experiment such that its risk
function (resp., the distribution of the likelihood ratio) coincides
with $\rr$ (resp. $\mu$). Now, if we start in (i) with the risk
function $\rr_\E$ of an experiment $\E$, we obtain a new experiment,
say, $\varkappa(\E)$, equivalent to $\E$. Moreover, experiments $\E_1$
and $\E_2$ are equivalent if and only if $\varkappa(\E_1)=\varkappa
(\E
_2)$. In other words, the rule $\E\rightsquigarrow\varkappa(\E)$ is a
representation of binary experiments. Another representation is given
in the proof of (iii).

\begin{defin}\label{jd6s3hsd}
Let $\E=(\varOmega,\cF,\PP,\PP')$ and $\widetilde{\E
}=(\widetilde{\varOmega
},\widetilde{\cF},\widetilde{\PP},\widetilde{\PP}')$ be two binary
experiments. $\E$ is called \textit{$\varepsilon$-deficient} with
respect to $\widetilde{\E}$ if for any $\widetilde{\varphi} \in
\varPhi
(\widetilde{\E})$ there is $\varphi\in\varPhi(\E)$ such that
$\alpha
(\varphi) \leq\alpha(\widetilde{\varphi}) + \varepsilon/ 2$ and
$\beta
(\varphi) \leq\beta(\widetilde{\varphi}) + \varepsilon/ 2$. The number
\[
\delta_2(\E, \, \widetilde{\E}) := \inf\, \{ \varepsilon\geq0
\colon
\E\text{ is $\varepsilon$-deficient with respect to $\widetilde{\E
}$} \} \text{}
\]
is called the (asymmetric) \textit{deficiency} of $\E$ with respect to
$\widetilde{\E}$. Define also the (symmetric) \textit{deficiency}
\[
\Delta_2(\E, \, \widetilde{\E}) : = \max\, \bigl(
\delta_2(\E, \, \widetilde{\E}), \, \delta_2(\widetilde{
\E}, \, \E) \bigr) \text{}
\]
between $\E$ and $\widetilde{\E}$.
\end{defin}

It is easy to check that $\E\succeq\widetilde{\E}$ if and only if
$\delta_2(\E, \, \widetilde{\E}) = 0$. Hence, $\E\sim\widetilde
{\E}$
if and only if $\Delta_2(\E, \, \widetilde{\E}) = 0$. It is also easy
to check that $\delta_2$ and $\Delta_2$ satisfy the triangle inequality
and, hence, $\Delta_2$ is a metric on the space of types of
experiments. We shall see after the next proposition that this metric
space is a compact space.

\begin{prop}
Let $\E$ and $\widetilde{\E}$ be binary experiments. The following
statements are equivalent\/{\rm:}
\begin{itemize}
\item[\emph{(i)}] $\E$ is $\varepsilon$-deficient with respect to
$\widetilde{\E}$.
\item[\emph{(ii)}] For all $u \in[0, \, 1 -\tfrac{\varepsilon}{2}]$,
\begin{equation}
\label{jyds5tg3hs7} \rr_{\E}\biggl(u + \frac{\varepsilon
}{2}\biggr) \leq
\rr_{\widetilde{\E}}(u) + \frac{\varepsilon}{2}.
\end{equation}
\item[\emph{(iii)}] For all $\pi\in(0,\,1)$,
\begin{equation}
\label{ntdrgf82hsj} \bb_{\E}(\pi) \leq\bb_{\widetilde{\E}}(\pi
) +
\frac{\varepsilon}{2}.
\end{equation}
\end{itemize}
\end{prop}

\begin{proof}
(i)$\;\Leftrightarrow\;$(ii) follows immediately from Definition~\ref
{jd6s3hsd}, so our goal is to prove (ii)$\;\Leftrightarrow\;$(iii)
using dual relations (\ref{kusd73hjds}) and (\ref{d78hjxxxxx76shs}). A
direct proof of (i)$\;\Leftrightarrow\;$(iii) can be found in \cite
{Strasser1985}.

To simplify the notation, put $\iqf:= \iqf_{Z}$, $\idf:= \idf_{Z}$,
while the corresponding functions in the experiment $\widetilde{\E}$
are denoted by $\widetilde{\iqf}$ and $\widetilde{\idf}$. Since
$\rr_{\E
}(u) = \iqf(1 - u)$ (\ref{jyds5tg3hs7}) is equivalent to
\begin{equation}
\label{uas7uhj2zq2} \iqf(u) \leq\widetilde{\iqf}\biggl(u + \frac
{\varepsilon}{2}\biggr) +
\frac
{\varepsilon}{2}\quad\text{for all}\ u \in\biggl[0, \, 1 -\frac
{\varepsilon}{2}
\biggr].
\end{equation}
In turn, if follows from (\ref{e:J2B}) that (\ref{ntdrgf82hsj}) is
equivalent to
\begin{equation}
\label{jyd6tgheh1} \idf(x) \geq\widetilde{\idf}(x) - \frac
{\varepsilon}{2} (1 + x)\quad\text
{for all}\ x > 0.
\end{equation}

Since $\idf(x) = 0$ for $x \leq0$, we have $\iqf(u) = \sup_{x \geq
0}\{
xu - \idf(x)\}$ and similarly for $\widetilde{\iqf}$. Thus, it follows
from (\ref{jyd6tgheh1}) that, for $u \geq0$,
\[
\iqf(u) \leq\sup_{x \geq0} \biggl\{xu + \frac{\varepsilon
}{2}(1+x) -
\widetilde{\idf}(x) \biggr\} = \widetilde{\iqf}\biggl(u + \frac
{\varepsilon
}{2}\biggr) +
\frac{\varepsilon}{2} \text{.}
\]

Conversely, let (\ref{uas7uhj2zq2}) hold true, and let $x > 0$ be such
that $F_{\widetilde{Z}}(x - 0) \geq\tfrac{\varepsilon}{2}$, where
$\widetilde{Z}$ is the Radon--Nikod\'ym derivative of the $\widetilde
{\PP}$-absolutely continuous part of $\widetilde{\PP}'$ with respect to
$\widetilde{\PP}$. Then
\begin{align*}
\idf(x) &\geq\sup_{u \in[0,\, 1]} \bigl\{xu - \iqf(u) \bigr\}
\geq\sup
_{u \in[0,\, 1 - \frac{\varepsilon}{2} ]} \bigl\{xu - \iqf(u)
\bigr\}
\\
&\geq\sup_{u \in[0,\, 1 - \frac{\varepsilon}{2} ]} \biggl\{ xu -
\frac{\varepsilon}{2} - \widetilde{
\iqf}\biggl(u + \frac{\varepsilon
}{2}\biggr) \biggr\} = \sup_{u \in[\frac{\varepsilon}{2}, \, 1 ]}
\biggl\{xu - \frac{(1+x)\varepsilon}{2} - \widetilde{\iqf}(u)
\biggr\}
\\
&= -\frac
{\varepsilon}{2}(x+1) + \widetilde{\idf}(x) \text{,}
\end{align*}
where the last equality follows from the fact that the supremum in
(\ref
{d78hjxxxxx76shs}) is attained at $u \in[F_{\widetilde{Z}}(x - 0), \,
F_{\widetilde{Z}}(x)]$, cf. \eqref{jsyajtv122sd}. It remains to note
that if $x$ is such that $F_{\widetilde{Z}}(x - 0) < \tfrac
{\varepsilon
}{2}$, then $\widetilde{\idf}(x) \leq\tfrac{\varepsilon x}{2}$ and
(\ref{jyd6tgheh1}) is obviously true.
\end{proof}

As a consequence, we obtain the following expressions for $\delta_2(\E,
\, \widetilde{\E})$ and $\Delta_2(\E, \, \widetilde{\E})$, see
(\cite
{Strasser1985}, \cite[p.~604]{Torgersen1991}):

\begin{cor}\label{c:Delta2b}
Let $\E$ and $\widetilde{\E}$ be binary experiments. Then
\begin{align*}
\delta_2(\E, \, \widetilde{\E}) &= \frac{1}{2}\sup
_{\pi\in[0,\,
1]} \bigl\{\bb_{\E}(\pi) - \bb_{\widetilde{\E}}(
\pi) \bigr\},
\\
\Delta_2(\E, \, \widetilde{\E}) &= \frac{1}{2}\sup
_{\pi\in[0,\,
1]} \big|\bb_{\E}(\pi) - \bb_{\widetilde{\E}}(\pi) \big| =
\frac{1}2 L(F,\widetilde F),
\end{align*}
where $L(\cdot,\cdot)$ is the L\'evy distance between distribution
functions{\rm,} $F$ is defined as in \eqref{e:repr1} with $\rr=\rr
_\E
$, and $\widetilde F$ is defined similarly with $\rr=\rr_{\widetilde
\E}$.
\end{cor}

The subset of concave functions $\bb$ on $[0, \,1]$ satisfying $0 \leq
\bb(\pi) \leq\pi\wedge(1 - \pi)$ is clearly closed with respect to
uniform convergence and is equicontinuous. By the Arzela--Ascoli
theorem, this subset is a compact in the space $C[0,\,1]$ with
sup-norm. Therefore, the space of types of experiments is a compact
metric space with $\Delta_2$-metric.

\begin{defin}
Let $\E=(\varOmega,\cF,\PP,\PP')$ and $\E^{n}=(\varOmega^{n},\cF
^{n},\PP
^{n},\PP'^{n})$, $n \geq1$, be binary experiments. We say that $\E^n$
\textit{weakly converges} to $\E$ if $\Delta_2(\E^{n}, \, \E)
\rightarrow0$ as $n \rightarrow\infty$.
\end{defin}

\begin{prop}
Let $\E$ and $\E^n$, $n\geq1$, be binary experiments. The following
statements are equivalent\/{\rm:}
\begin{itemize}
\item[\emph{(i)}] $\Delta_2(\E^{n}, \, \E) \rightarrow0$.
\item[\emph{(ii)}] $\rr_{\E^{n}}$ converges to $\rr_{\E}$
pointwise on
$(0,\,1]$.
\item[\emph{(ii$'$)}] $\rr_{\E^{n}}$ converges uniformly to $\rr
_{\E}$
on any $[a,\,1] \subset(0,\,1]$.
\item[\emph{(iii)}] $\bb_{\E^{n}}$ converges uniformly to $\rr_{\E}$
on $[0,\,1]$.
\item[\emph{(iv)}] $\mu_{\E^{n}}$ weakly converges to $\mu_{\E}$.
\end{itemize}
\end{prop}

\begin{proof}
The equivalences (i)$\;\Leftrightarrow\;$(ii) and (i)$\;
\Leftrightarrow
\;$(iii) follow from Corollary~\ref{c:Delta2b}, and the equivalence of
(ii), (ii$'$), and (iv) is a consequence of Theorem~\ref{jny6hgbhsjtr}
and Proposition~\ref{jydy3gh4hsj}. However, we prefer to give a direct
proof of the equivalence (i)$\;\Leftrightarrow\;$(ii) without using the
L\'evy distance.

Assume (i). By (\ref{jyds5tg3hs7}),
\begin{align*}
&\rr_{\E^{n}}(u) + \frac{\Delta_2(\E^{n}, \, \E)}{2}\geq\rr_{\E
} \biggl(u +
\frac{\Delta_2(\E^{n}, \, \E)}{2} \biggr) , \qquad0 \leq u \leq1
- \frac{\Delta_2(\E^{n}, \, \E)}{2},\\
&\rr_{\E} \biggl(u - \frac{\Delta_2(\E^{n}, \, \E)}{2} \biggr)
+ \frac
{\Delta_2(\E^{n}, \, \E)}{2}\geq
\rr_{\E^{n}}(u) , \qquad\frac
{\Delta_2(\E^{n}, \, \E)}{2} \leq u \leq1.
\end{align*}
Passing to the limit as $n \rightarrow\infty$, we get
\begin{align*}
&\liminf_{n \rightarrow\infty}\rr_{\E^{n}}(u) \geq\rr_{\E}(u)
\quad\text{for}\ 0 \leq u < 1,\\
&\rr_{\E}(u) \geq\limsup_{n \rightarrow\infty}\rr_{\E^{n}}(u)
\quad\text{for}\ 0 < u \leq1.
\end{align*}
Combining these inequalities, we obtain $\lim_{n \rightarrow\infty
}\rr
_{\E^{n}}(u) = \rr_{\E}(u)$ for $0 < u < 1$. Since risk functions
vanish at $1$, the convergence holds for $u = 1$ as well.

Now the converse implication (ii)$\;\Rightarrow\;$(i) is proved by
standard compactness arguments.
\end{proof}

\section{Chacon--Walsh revisited}\label{s:sec5}

The Skorokhod embedding problem was posed and solved by Skorokhod \cite
{Skorokhod:65} in the following form: given a centered distribution
$\mu
$ with finite second moment, find a stopping time $T$ such that $\EE[T]
< \infty$ and $\operatorname{Law} (B_T) = \mu$, where $B = (B_t)_{t
\geq0}$, $B_0=0$, is a standard Brownian motion. Chacon and Walsh
\cite
{ChaconWalsh1976} suggest to construct $T$ as the limit of an
increasing sequence of stopping times $T_n$, each being the first exit
time (after the previous one) of $B$ from a compact interval. This
construction has a simple graphical interpretation in terms of the
potential functions of $B_{T_n}$ (we recall that potential functions
are defined in (\ref{ew6ghzjgaj})).

Cox \cite{Cox2008} extends the Chacon--Walsh construction to a more
general case. He considers a Brownian motion $B = (B_t)_{t \geq0}$
with a given integrable starting distribution $\mu_0$ for $B_0$ and a
general integrable target distribution $\mu$. A solution $T$ (such that
$\operatorname{Law} (B_T)= \mu$) must be found in the class of \emph
{minimal\/} stopping times.

It is easy to observe that the Chacon--Walsh construction has a
graphical interpretation in terms of integrated quantile functions as
well; moreover, in our opinion, the picture is more simple. We give
alternative proofs of the result in \cite{ChaconWalsh1976} and of some
results in \cite{Cox2008}. Moreover, we construct a minimal stopping
time in some special case where $\mu_0$ and $\mu$ may be non-integrable.

Let us recall the definition of the balayage.
For a probability measure $\mu$ on $\mathbb{R}$ and an interval $I =
(a,\,b)$, $-\infty< a < b < +\infty$, the balayage $\mu_{I}$ of $\mu$
on $I$ is defined as the measure which coincides with $\mu$ outside
$[a,\,b]$, vanishes on $(a,\,b)$, and such that
\begin{equation}
\label{htb54ga5sf1} \mu_I\bigl(\{a\}\bigr) = \int_{[a,\,b]}
\frac{b-x}{b-a}\,\mu(dx) , \qquad\mu_I\bigl(\{b\}\bigr) =
\int_{[a,\,b]}\frac{x-a}{b-a}\,\mu(dx) \text{.}
\end{equation}
Since
\begin{equation}
\label{e:balayage} \int_{[a,\,b]}\mu_I(dx) = \int
_{[a,\,b]}\mu(dx) \quad\text{and}\quad\int_{[a,\,b]}x
\,\mu_I(dx) = \int_{[a,\,b]}x\,\mu(dx),
\end{equation}
the balayage $\mu_I$ is a probability measure and
has the same mean as $\mu$ (if defined). It follows that, if $B =
(B_t)_{t \geq0}$ is a continuous local martingale with $\langle
B,B\rangle_\infty= \infty$ a.s. (e.\,g. a Brownian motion), $\mu$ is
the distribution of $B_S$, where $S$ is a stopping time, and the
stopping time $T$ is defined by\vadjust{\goodbreak}
\begin{equation}
\label{unsayu712hjgsd} T := \inf\{t \geq S \colon B_t \notin I \}
\text{,}
\end{equation}
then $T < +\infty$ a.\,s. and the distribution of $B_T$ is the balayage
$\mu_I$.

Let $X$ and $Y$ be random variables with the distributions $\mu$ and
$\mu_I$ respectively. It is clear that
\[
q_Y^{L}(u) = \lleft\{ %
\begin{array}{@{}ll}
q_X^{L}(u), & \hbox{if\ $0<u \leq F_X(a-0)$ or $F_X(b)<u<1$,} \\
a, & \hbox{if\ $F_X(a-0)<u\leq F_X(a-0)+\mu_I(\{a\})$,} \\
b, & \hbox{if\ $F_X(a-0)+\mu_I(\{a\})<u\leq F_X(b)$.}
\end{array}
\rright.
\]
Moreover, the second equality in \eqref{e:balayage} can be rewritten as
\[
\iqf_X\bigl(F_X(b)\bigr)-\iqf_X
\bigl(F_X(a-0)\bigr) = \iqf_X\bigl(F_Y(b)
\bigr)-\iqf_X\bigl(F_Y(a-0)\bigr) \text{.}
\]
This allows us to describe how to obtain the integrated quantile
function of $Y$: pass the tangent lines with the slopes $a$ and $b$ to
the graph of $\iqf_X$, replace the curve on this graph between points
where the graph meets the lines by the corresponding segments of these
lines. If the point of intersection of these lines lies below the
horizontal axis, then shift the resulting graph vertically upwards so
that this point will come on the horizontal axis.

\begin{figure}[t]
\includegraphics{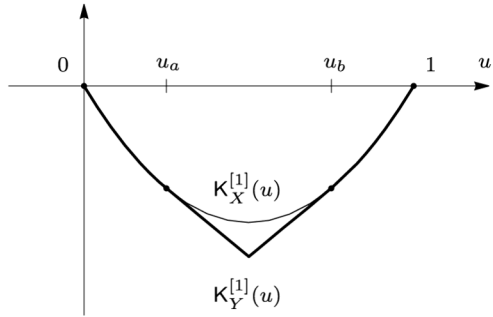}
\caption{Graphs of shifted integrated quantile functions $\iqf_X^{[1]}$
and $\iqf_Y^{[1]}$: the distribution of $Y$ is the balayage of the
distribution of $X$}\label{f_ahgsd}
\end{figure}

If $\EE[X^+] <\infty$ (resp. $\EE[X^-] <\infty$), then $\EE[Y^+]
<\infty
$ (resp. $\EE[Y^-] <\infty$), and the last step is not needed if we
deal with shifted integrated quantile functions $\iqf_X^{[1]}$ and
$\iqf
_Y^{[1]}$ (resp. $\iqf_X^{[0]}$ and $\iqf_Y^{[0]}$), see Fig.~\ref
{f_ahgsd}. We state this fact in the following lemma only in the case
where $\EE[X^+] <\infty$. Its proof is immediate from the previous paragraph.

\begin{lemma}\label{bh6ghhst612}
Let $\mu$ be the distribution of a random variable $X$ with $\EE[X^+] <
\infty$, $\operatorname{Law} (Y)= \mu_I$, where $I = (a,\,b)$ is a
finite interval. Put $u_a:=F_X(a-0)$ and $u_b:=F_X(b)$. Then $\EE[Y^+]
< \infty$ and
\begin{equation*}
\iqf_Y^{[1]}(u) = \lleft\{ %
\begin{array}{@{}ll}
\iqf_X^{[1]}(u), & \hbox{if $u\notin(u_a,\,u_b)$,} \\
\bigl(a (u - u_a) + \iqf_X^{[1]}(u_a) \bigr)\vee\bigl(b (u -
u_b) + \iqf_X^{[1]}(u_b) \bigr), & \hbox{if
$u\in(u_a,\,u_b)$.}
\end{array}
\rright.
\end{equation*}
In particular{\rm,} $\iqf_Y^{[1]}(u) \leq\iqf_X^{[1]}(u)$ for all $u
\in[0,\,1]$.
\end{lemma}

The next lemma is a key tool in our future construction.\vadjust
{\goodbreak}

\begin{lemma}\label{l:step}
Let $X$ and $Y$ be random variables such that $\EE[X^+] <\infty$,
$\EE
[Y^+] <\infty$, and $\iqf_Y^{[1]}(u) \leq\iqf_X^{[1]}(u)$ for all $u
\in[0,\,1]$. Fix $v \in(0,\,1)$. Then there is a random variable $Z$
such that $\iqf_Y^{[1]}(u) \leq\iqf_Z^{[1]}(u) \leq\iqf_X^{[1]}(u)$
for all $u \in[0,\,1]$, $\iqf_Z^{[1]}(v) = \iqf_Y^{[1]}(v)$, and the
distribution of $Z$ is a balayage of the distribution of~$X$.
\end{lemma}

\begin{proof}
Without loss of generality, we may assume that $\iqf_X^{[1]}(v) > \iqf
_Y^{[1]}(v)$. Let us consider the following equation:
\begin{equation}
\label{jyas5sjah7112sa} xv - \idf_X(x) - \EE\bigl[X^{+}\bigr] =
\iqf_Y^{[1]}(v) \text{.}
\end{equation}
The maximum of the left-hand side over $x$ equals $\iqf_X(v) - \EE
[X^{+}] = \iqf_X^{[1]}(v)$ and is greater than the right-hand side.
Moreover, it is attained at $x\in[q_X^{L}(v),\,q_X^{R}(v)]$. Further,
applying Theorem~\ref{ystwqe52dgh}~(iv)--(v), we get
\[
\lim_{x \rightarrow+ \infty}\bigl(xv - \idf_X(x)\bigr) = \lim
_{x \rightarrow+
\infty}\bigl(x - \idf_X(x)\bigr) + \lim
_{x \rightarrow+ \infty}(v-1)x = -\infty
\]
and
\[
\lim_{x \rightarrow- \infty}\bigl(xv - \idf_X(x)\bigr) = \lim
_{x \rightarrow-
\infty}x \biggl(v -\frac{\idf_X(x)}{x} \biggr) = -\infty\text{.}
\]
Since the left-hand side of (\ref{jyas5sjah7112sa}) is a concave
function in $x$, the equation~(\ref{jyas5sjah7112sa}) has two solutions
$a<q_X^{L}(v)$ and $b>q_X^{R}(v)$, i.\,e. $F_X(a) < v < F_X(b-0)$.

Using Corollary~\ref{bsat65gh2hjs}, rewrite equation~(\ref
{jyas5sjah7112sa}) in the form
\[
\iqf_X^{[1]}\bigl(F_X(x)\bigr) =
\iqf_Y^{[1]}(v) + x \bigl(F_X(x) - v\bigr)
\text{.}
\]
This equality for $x=a$ (resp. $x=b$) says that the straight line with
the slope $a$ (resp. $b$) and passing through the point $(v, \, \iqf
_Y^{[1]}(v))$ meets the curve $\iqf_X^{[1]}$ at the point where the
first coordinate is $F_X(a)$ (resp. $F_X(b)$). Due to (\ref
{jsyajtv122sd}), these straight lines are tangent lines to the curve
$\iqf_X^{[1]}$. Comparing with Lemma~\ref{bh6ghhst612}, we obtain that
a random variable $Z$ such that its distribution is the balayage of the
distribution of $X$ on $I = (a,\,b)$ satisfies all the requirements.
\end{proof}

From now on, we assume that there is a probability space with
filtration $(\varOmega,\cF,\break(\cF_t)_{t\geq0},\PP)$ and
an $(\cF_t,\PP)$-Brownian motion $B=(B_t)_{t\geq0}$ with an arbitrary
initial distribution. For $c>0$, let
\[
H_c=\inf\{t\geq0\colon B_t\geq c\}.
\]
The next lemma is inspired by Theorem~5 in \cite{CoxHobson2006}.

\begin{lemma}\label{l:CH} Let $S$ be a stopping time and $T$ defined by
\eqref{unsayu712hjgsd} with $I=(a,b)$. If $\EE[B_S^+]<\infty$ and
$c\PP
(S\geq H_c)\leq\EE[B_S \mathbb{1}_{\{S \geq H_c\}}]$, then $\EE
[B_T^+]<\infty$ and $c\PP(T\geq H_c)\leq\EE[B_T \mathbb{1}_{\{T
\geq
H_c\}}]$.
\end{lemma}

\begin{proof}
By the strong Markov property, in view of boundedness of the random
variables under the conditional expectations below,
\[
\EE[B_T-B_S|\cF_S] =0 \quad\text{and}
\quad\EE[B_T-B_{(S\vee H_c)\wedge
T}|\cF_{(S\vee H_c)\wedge T}] =0.
\]
Since $\{S\geq H_c\}\in\cF_S$ and $\{S< H_c \leq T\}\in\cF_{(S\vee
H_c)\wedge T}$, we get
\begin{align*}
c\PP(T\geq H_c) &= c\PP(S\geq H_c) + c\PP(T\geq
H_c>S)
\\
& \leq\EE[B_S \mathbb{1}_{\{S \geq H_c\}}] + \EE[B_{(S\vee
H_c)\wedge T}
\mathbb{1}_{\{S < H_c \leq T\}}]
\\
&= \EE[B_T \mathbb{1}_{\{S \geq H_c\}}] + \EE[B_{T}
\mathbb{1}_{\{S <
H_c \leq T\}}] = \EE[B_T \mathbb{1}_{\{T \geq H_c\}}].\qedhere
\end{align*}
\end{proof}

Let us also recall that $T$ is a \textit{minimal} stopping time if any
stopping time $R \leq T$ with $\operatorname{Law} (B_R) =
\operatorname
{Law} (B_T)$ satisfies $R = T$ a.\,s.

\begin{thm}\label{t:ChaconWalsh}
Let $\mu_0$ and $\mu$ be distributions on $\bbR$ such that $\int
_{\bbR}
x^+\,\mu(dx)<\infty$ and
\begin{equation*}
\label{e:icx} \int_{\bbR} (x-y)^+\,\mu_0(dx)\leq
\int_{\bbR} (x-y)^+\,\mu(dx) \quad\text{for all}\ y\in\bbR.
\end{equation*}
Let $B$ be a Brownian motion with the initial distribution
$\operatorname{Law} (B_0)= \mu_0$. Then there is an increasing sequence
of stopping times $0=T_0\leq T_1\leq\cdots\leq T_n\leq\dots$ such that
$T:=\lim_{n\to\infty} T_n$ is a minimal a.\,s. finite stopping time,
the distribution of $B_{T_n}$ is a balayage of the distribution of
$B_{T_{n-1}}$ for each $n=1,2,\dots$, and $\operatorname{Law} (B_T)=
\mu$.
\end{thm}

\begin{proof}
Put $X_0:= B_{T_0}$ and let $\operatorname{Law} (Y)= \mu$. Then $\iqf
_Y^{[1]}(u) \leq\iqf_{X_0}^{[1]}(u)$ for all $u \in[0,\,1]$. Take an
arbitrary sequence $\{v_n\}$ of distinct points in $(0,\,1)$ such that
$\{v_n\colon n=1,\,2,\,\dots\}$ is dense in $[0,\,1]$. Recursively
define $X_n$ as $Z$ in Lemma~\ref{l:step} applied to $X=X_{n-1}$, $Y$,
and $v=v_n$. Then we obtain a sequence $\{X_n\}$ of random variables
such that
\[
\iqf_Y^{[1]}(u) \leq\iqf_{X_n}^{[1]}(u)
\leq\iqf_{X_{n-1}}^{[1]}(u)\leq\iqf_{X_0}^{[1]}(u)
, \quad u \in(0,\,1],
\]
and $\iqf_Y^{[1]}(v_n) = \iqf_{X_n}^{[1]}(v_n)$, which implies $\iqf
_Y^{[1]}(v_n) = \iqf_{X_m}^{[1]}(v_n)$ for all $n$ and $m\geq n$. Then
$\lim_{n\to\infty}\iqf_{X_n}^{[1]}(u)$ exists, is finite for all $u
\in
(0,\,1)$, and coincides with $\iqf_Y^{[1]}(u)$ on the set $\{v_n\colon
n=1,\,2,\,\dots\}$. Being a convex function in $u$, this limit
coincides with $\iqf_Y^{[1]}(u)$ everywhere on $(0,\,1)$. It follows
from Remark~\ref{dsfse3sa} that $X_n$ weakly converges to $Y$ and the
sequence $\{X_n^+\}$ is uniformly integrable.

Moreover, the construction in Lemma~\ref{l:step} provides an interval
$(a,b)$ denoted by $I_n$ such that the distribution of $X_n$ is the
balayage of the distribution of $X_{n-1}$ on $I_n$. Now recursively define
\[
T_n := \inf\{t \geq T_{n-1} \colon B_t \notin
I_n \} \text{.}
\]
Then $B_{T_n}$ has the same distribution as $X_n$. Since
\[
\EE[B_0 \mathbb{1}_{\{0 \geq H_c\}}] = \EE[B_0
\mathbb{1}_{\{B_0 \geq
c\}}] \geq c\PP(B_0\geq c) \text{,}
\]
we conclude from Lemma~\ref{l:CH} that, for any $c>0$ and $n$,
\[
c\PP(T_n\geq H_c)\leq\EE[B_{T_n}
\mathbb{1}_{\{T_n \geq H_c\}}] \leq\EE\bigl[B_{T_n}^+\bigr] \leq
\EE
\bigl[Y^+\bigr].
\]
If $\PP(T=\infty)=\delta>0$, then the limit of the expression on the
left in the last inequality is greater than or is equal to $c\delta$,
which is greater than the right-hand side if $c$ is large enough. This
contradiction proves that $T<\infty$ a.s. This implies that $B_{T_n}$
converges a.s. to $B_T$ and, hence, $\operatorname{Law} (B_T)= \mu$.

It remains to prove that $T$ is a minimal stopping time. According to
Theorem 4.1 in \cite{GushchinUrusov2016}, it is enough to find a
one-to-one function $G$ such that $G(B)^{T}$ is a closed submartingale.

Let $g(x)$, $x\in\bbR$, be a continuously differentiable function with
the following properties: $g\equiv1$ on $[0,+\infty)$ and is strictly
positive and increasing on $(-\infty,0]$, $\int_{-\infty}^0 g(x)\,dx
<\infty$, and $g'(x)\leq1$ for all $x$. Put $G(y):=\int_0^y g(x)\,dx$,
then, in particular, $G$ is strictly increasing and bounded from below,
and $G(y)=y$ for $y\geq0$. By It\^o's formula,
\[
G(B_t)=G(B_0)+\int_0^t
g(B_s)\,dB_s + \frac{1}{2}\int
_0^t g'(s)\, ds=:G(B_0)
+ M_t +A_t \text{,}
\]
where $G(B_0)$ is an integrable random variable, $M$ is a local
martingale and $[M,M]_t = \int_0^t g^2(B_s)\,ds\leq t $. Hence, by the
Burkholder--Davis--Gundy inequality, $\EE\sup_{s\leq t}|M_s|$ is
integrable; in particular, $M$ is a martingale. Finally, $A$ is an
increasing process and $A_t\leq t/2$. Therefore, $G(B)$ is a
submartingale and, hence, so are the stopped processes $G(B)^T$ and
$G(B)^{T_n}$. Note that, by construction, the process $(B-B_0)^{T_n}$
is bounded (by the sum of the lengths of $I_k$, $k\leq n$) for a fixed
$n$. Hence, $G(B_{t\wedge T_n}) \leq B_{t\wedge T_n}^+ \leq B_0^+ +
\sup_{s\leq T_n}|B_s-B_0|$. We conclude that the submartingale $G(B)^{T_n}$
is uniformly integrable, hence, $G(B_{t\wedge T_n})\leq_{icx}
G(B_{T_n})$ for any $n$ and $t$. On the other hand, $B_{T_n}^+ \leq
_{icx} B_T^+$. Combining, we get $[G(B_{t\wedge T_n})]^+\leq_{icx}
B_T^+$ for any $n$ and $t$. We can pass to the limit as $n\to\infty$ in
this inequality, which shows that the family $[G(B_{t\wedge T})]^+$,
$t\in\bbR_+$, is uniformly integrable. The claim follows.
\end{proof}

\begin{remark}
It has been already mentioned in the proof of Theorem~\ref
{jhdt67thdbba} that the assumptions on $\mu_0$ and $\mu$ in
Theorem~\ref
{t:ChaconWalsh} are equivalent to $\mu_0 \leq_{icx} \mu$.
\end{remark}

\begin{remark}
Let $\mu_0$ and $\mu$ satisfy the assumptions of Theorem~\ref
{t:ChaconWalsh}. As a by-product, we have obtained the following
classical characterization of increasing convex order: there exist
random variables $X_0$ and $X$ defined on the same probability space
such that $\operatorname{Law} (X_0 )= \mu_0$, $\operatorname{Law} (X)=
\mu$ and
\[
\EE[X|X_0] \geq X_0 \text{.}
\]
Indeed, take $X_0 = B_0$ and $X = B_T$ and use the uniform
integrability of $(B_{T_n}^{+})_{n \geq1}$ to obtain the desired inequality.
\end{remark}

\section*{Acknowledgments}
We thank three anonymous referees for constructive comments and remarks
that helped improving the exposition.

\bibliographystyle{vmsta-mathphys}

\begin{thebibliography}{28}

\bibitem{Blackwell1951}
%
\begin{bchapter}
\bauthor{\bsnm{Blackwell}, \binits{D.}}:
\bctitle{Comparison of experiments}.
In: \bbtitle{Proceedings of the {S}econd {B}erkeley {S}ymposium on
{M}athematical {S}tatistics and {P}robability, 1950},
pp.~\bfpage{93}--\blpage{102}.
\bpublisher{University of California Press, Berkeley and Los Angeles}
(\byear{1951}).
\MR{0046002}
\end{bchapter}
%
\OrigBibText
%
\begin{bchapter}
\bauthor{\bsnm{Blackwell}, \binits{D.}}:
\bctitle{Comparison of experiments}.
In: \bbtitle{Proceedings of the {S}econd {B}erkeley {S}ymposium on
{M}athematical {S}tatistics and {P}robability, 1950},
pp.~\bfpage{93}--\blpage{102}.
\bpublisher{University of California Press, Berkeley and Los Angeles}
(\byear{1951}).
\MR{0046002}
\end{bchapter}
%
\endOrigBibText
\bptok{structpyb}%
\endbibitem

\bibitem{Blackwell1953}
%
\begin{barticle}
\bauthor{\bsnm{Blackwell}, \binits{D.}}:
\batitle{Equivalent comparisons of experiments}.
\bjtitle{Ann. Math. Statistics}
\bvolume{24},
\bfpage{265}--\blpage{272}
(\byear{1953}).
\MR{0056251}
\end{barticle}
%
\OrigBibText
%
\begin{barticle}
\bauthor{\bsnm{Blackwell}, \binits{D.}}:
\batitle{Equivalent comparisons of experiments}.
\bjtitle{Ann. Math. Statistics}
\bvolume{24},
\bfpage{265}--\blpage{272}
(\byear{1953}).
\MR{0056251}
\end{barticle}
%
\endOrigBibText
\bptok{structpyb}%
\endbibitem

\bibitem{CarlierChernozhukovGalichon2016}
%
\begin{barticle}
\bauthor{\bsnm{Carlier}, \binits{G.}},
\bauthor{\bsnm{Chernozhukov}, \binits{V.}},
\bauthor{\bsnm{Galichon}, \binits{A.}}:
\batitle{Vector quantile regression: An optimal transport approach}.
\bjtitle{Ann. Statist.}
\bvolume{44}(\bissue{3}),
\bfpage{1165}--\blpage{1192}
(\byear{2016}).
\bid{doi={10.1214/\\15-AOS1401}, mr={3485957}}
\end{barticle}
%
\OrigBibText
%
\begin{barticle}
\bauthor{\bsnm{Carlier}, \binits{G.}},
\bauthor{\bsnm{Chernozhukov}, \binits{V.}},
\bauthor{\bsnm{Galichon}, \binits{A.}}:
\batitle{Vector quantile regression: An optimal transport approach}.
\bjtitle{Ann. Statist.}
\bvolume{44}(\bissue{3}),
\bfpage{1165}--\blpage{1192}
(\byear{2016}).
\end{barticle}
%
\endOrigBibText
\bptok{structpyb}%
\endbibitem

\bibitem{ChaconWalsh1976}
%
\begin{bchapter}
\bauthor{\bsnm{Chacon}, \binits{R.V.}},
\bauthor{\bsnm{Walsh}, \binits{J.B.}}:
\bctitle{One-dimensional potential embedding}.
In: \bbtitle{S\'eminaire de {P}robabilit\'es, {X}}.
\bsertitle{Lecture Notes in Math.},
vol.~\bseriesno{511},
pp.~\bfpage{19}--\blpage{23}.
\bpublisher{Springer}
(\byear{1976}).
\bid{mr={0445598}}
\end{bchapter}
%
\OrigBibText
%
\begin{bchapter}
\bauthor{\bsnm{Chacon}, \binits{R.V.}},
\bauthor{\bsnm{Walsh}, \binits{J.B.}}:
\bctitle{One-dimensional potential embedding}.
In: \bbtitle{S\'eminaire de {P}robabilit\'es, {X}}.
\bsertitle{Lecture Notes in Math.},
vol.~\bseriesno{511},
pp.~\bfpage{19}--\blpage{23}.
\bpublisher{Springer}
(\byear{1976}).
\MR{0445598}
\end{bchapter}
%
\endOrigBibText
\bptok{structpyb}%
\endbibitem

\bibitem{ChernozhukovGalichonHallin2017}
%
\begin{barticle}
\bauthor{\bsnm{Chernozhukov}, \binits{V.}},
\bauthor{\bsnm{Galichon}, \binits{A.}},
\bauthor{\bsnm{Hallin}, \binits{M.}},
\bauthor{\bsnm{Henry}, \binits{M.}}:
\batitle{Monge--{K}antorovich depth, quantiles, ranks and signs}.
\bjtitle{Ann. Statist.}
\bvolume{45}(\bissue{1}),
\bfpage{223}--\blpage{256}
(\byear{2017}).
\bid{doi={\\10.1214/16-AOS1450}, mr={3611491}}
\end{barticle}
%
\OrigBibText
%
\begin{barticle}
\bauthor{\bsnm{Chernozhukov}, \binits{V.}},
\bauthor{\bsnm{Galichon}, \binits{A.}},
\bauthor{\bsnm{Hallin}, \binits{M.}},
\bauthor{\bsnm{Henry}, \binits{M.}}:
\batitle{Monge--{K}antorovich depth, quantiles, ranks and signs}.
\bjtitle{Ann. Statist.}
\bvolume{45}(\bissue{1}),
\bfpage{223}--\blpage{256}
(\byear{2017}).
doi:\doiurl{10.1214/16-AOS1450}
\end{barticle}
%
\endOrigBibText
\bptok{structpyb}%
\endbibitem

\bibitem{Cox2008}
%
\begin{bchapter}
\bauthor{\bsnm{Cox}, \binits{A.M.G.}}:
\bctitle{Extending {C}hacon-{W}alsh: minimality and generalised starting
distributions}.
In: \bbtitle{S\'eminaire de Probabilit\'es {XLI}}.
\bsertitle{Lecture Notes in Math.},
vol.~\bseriesno{1934},
pp.~\bfpage{233}--\blpage{264}.
\bpublisher{Springer}
(\byear{2008}).
\bid{doi={10.1007/978-3-540-77913-1\_12}}
\end{bchapter}
%
\OrigBibText
%
\begin{bchapter}
\bauthor{\bsnm{Cox}, \binits{A.M.G.}}:
\bctitle{Extending {C}hacon-{W}alsh: minimality and generalised starting
distributions}.
In: \bbtitle{S\'eminaire de Probabilit\'es {XLI}}.
\bsertitle{Lecture Notes in Math.},
vol.~\bseriesno{1934},
pp.~\bfpage{233}--\blpage{264}.
\bpublisher{Springer}
(\byear{2008}).
doi:\doiurl{10.1007/978-3-540-77913-1\_12}
\end{bchapter}
%
\endOrigBibText
\bptok{structpyb}%
\endbibitem

\bibitem{CoxHobson2006}
%
\begin{barticle}
\bauthor{\bsnm{Cox}, \binits{A.M.G.}},
\bauthor{\bsnm{Hobson}, \binits{D.G.}}:
\batitle{Skorokhod embeddings, minimality and non-centred target
distributions}.
\bjtitle{Probab. Theory Related Fields}
\bvolume{135}(\bissue{3}),
\bfpage{395}--\blpage{414}
(\byear{2006}).
\bid{doi={10.1007/\\s00440-005-0467-y}}
\end{barticle}
%
\OrigBibText
%
\begin{barticle}
\bauthor{\bsnm{Cox}, \binits{A.M.G.}},
\bauthor{\bsnm{Hobson}, \binits{D.G.}}:
\batitle{Skorokhod embeddings, minimality and non-centred target
distributions}.
\bjtitle{Probab. Theory Related Fields}
\bvolume{135}(\bissue{3}),
\bfpage{395}--\blpage{414}
(\byear{2006}).
doi:\doiurl{10.1007/s00440-005-0467-y}
\end{barticle}
%
\endOrigBibText
\bptok{structpyb}%
\endbibitem

\bibitem{EmbrechtsHofert2013}
%
\begin{barticle}
\bauthor{\bsnm{Embrechts}, \binits{P.}},
\bauthor{\bsnm{Hofert}, \binits{M.}}:
\batitle{A note on generalized inverses}.
\bjtitle{Math. Methods Oper. Res.}
\bvolume{77}(\bissue{3}),
\bfpage{423}--\blpage{432}
(\byear{2013}).
\bid{doi={10.1007/s00186-013-0436-7}}
\end{barticle}
%
\OrigBibText
%
\begin{barticle}
\bauthor{\bsnm{Embrechts}, \binits{P.}},
\bauthor{\bsnm{Hofert}, \binits{M.}}:
\batitle{A note on generalized inverses}.
\bjtitle{Math. Methods Oper. Res.}
\bvolume{77}(\bissue{3}),
\bfpage{423}--\blpage{432}
(\byear{2013}).
doi:\doiurl{10.1007/s00186-013-0436-7}
\end{barticle}
%
\endOrigBibText
\bptok{structpyb}%
\endbibitem

\bibitem{FaugerasRuschendorf2017}
%
\begin{barticle}
\bauthor{\bsnm{Faugeras}, \binits{O.}},
\bauthor{\bsnm{R\"uschendorf}, \binits{L.}}:
\batitle{Markov morphisms: a combined copula and mass transportation approach
to multivariate quantiles}.
\bjtitle{Mathematica Applicanda}
\bvolume{45}(\bissue{1}),
\bfpage{21}--\blpage{63}
(\byear{2017})
\end{barticle}
%
\OrigBibText
%
\begin{barticle}
\bauthor{\bsnm{Faugeras}, \binits{O.}},
\bauthor{\bsnm{R\"uschendorf}, \binits{L.}}:
\batitle{Markov morphisms: a combined copula and mass transportation approach
to multivariate quantiles}.
\bjtitle{Mathematica Applicanda}
\bvolume{45}(\bissue{1}),
\bfpage{21}--\blpage{63}
(\byear{2017})
\end{barticle}
%
\endOrigBibText
\bptok{structpyb}%
\endbibitem

\bibitem{Feller1971}
%
\begin{bbook}
\bauthor{\bsnm{Feller}, \binits{W.}}:
\bbtitle{An Introduction to Probability Theory and Its Applications. {V}ol.
{II}}.
\bedition{Second edition}.
\bpublisher{John Wiley \& Sons, Inc.},
\blocation{New York-London-Sydney}
(\byear{1971}).
\bid{mr={0270403}}
\end{bbook}
%
\OrigBibText
%
\begin{bbook}
\bauthor{\bsnm{Feller}, \binits{W.}}:
\bbtitle{An Introduction to Probability Theory and Its Applications. {V}ol.
{II}}.
\bsertitle{Second edition}.
\bpublisher{John Wiley \& Sons, Inc., New York-London-Sydney}
(\byear{1971}).
\MR{0270403}
\end{bbook}
%
\endOrigBibText
\bptok{structpyb}%
\endbibitem

\bibitem{FollmerSchied2011}
%
\begin{bbook}
\bauthor{\bsnm{F\"ollmer}, \binits{H.}},
\bauthor{\bsnm{Schied}, \binits{A.}}:
\bbtitle{Stochastic Finance: An Introduction in Discrete Time},
\bedition{4}th edn.
\bpublisher{Walter de Gruyter \& Co.},
\blocation{Berlin}
(\byear{2016})
\end{bbook}
%
\OrigBibText
%
\begin{bbook}
\bauthor{\bsnm{F\"ollmer}, \binits{H.}},
\bauthor{\bsnm{Schied}, \binits{A.}}:
\bbtitle{Stochastic Finance: An Introduction in Discrete Time},
\bedition{4}th edn.
\bpublisher{Walter de Gruyter \& Co., Berlin}
(\byear{2016})
\end{bbook}
%
\endOrigBibText
\bptok{structpyb}%
\endbibitem

\bibitem{GushchinUrusov2016}
%
\begin{barticle}
\bauthor{\bsnm{Gushchin}, \binits{A.A.}},
\bauthor{\bsnm{Urusov}, \binits{M.A.}}:
\batitle{Processes that can be embedded in a geometric {B}rownian motion}.
\bjtitle{Theory of Probability \& Its Applications}
\bvolume{60}(\bissue{2}),
\bfpage{246}--\blpage{262}
(\byear{2016}).
\bid{doi={\\10.1137/S0040585X97T987594}}
\end{barticle}
%
\OrigBibText
%
\begin{barticle}
\bauthor{\bsnm{Gushchin}, \binits{A.A.}},
\bauthor{\bsnm{Urusov}, \binits{M.A.}}:
\batitle{Processes that can be embedded in a geometric {B}rownian motion}.
\bjtitle{Theory of Probability \& Its Applications}
\bvolume{60}(\bissue{2}),
\bfpage{246}--\blpage{262}
(\byear{2016}).
doi:\doiurl{10.1137/S0040585X97T987594}
\end{barticle}
%
\endOrigBibText
\bptok{structpyb}%
\endbibitem

\bibitem{Hallin2017}
%
\begin{botherref}
\oauthor{\bsnm{Hallin}, \binits{M.}}:
On distribution and quantile functions, ranks and signs in $\mathbb{R}^d$.
ECARES Working Paper 2017-34
(2017)
\end{botherref}
%
\OrigBibText
%
\begin{botherref}
\oauthor{\bsnm{Hallin}, \binits{M.}}:
On distribution and quantile functions, ranks and signs in $\mathbb{R}^d$.
ECARES Working Paper 2017-34
(2017)
\end{botherref}
%
\endOrigBibText
\bptok{structpyb}%
\endbibitem

\bibitem{HardyLittlewood1930}
%
\begin{barticle}
\bauthor{\bsnm{Hardy}, \binits{G.H.}},
\bauthor{\bsnm{Littlewood}, \binits{J.E.}}:
\batitle{A maximal theorem with function-theoretic applications}.
\bjtitle{Acta Math.}
\bvolume{54}(\bissue{1}),
\bfpage{81}--\blpage{116}
(\byear{1930}).
\bid{doi={10.1007/BF02547518},mr={1555303}}
\end{barticle}
%
\OrigBibText
%
\begin{barticle}
\bauthor{\bsnm{Hardy}, \binits{G.H.}},
\bauthor{\bsnm{Littlewood}, \binits{J.E.}}:
\batitle{A maximal theorem with function-theoretic applications}.
\bjtitle{Acta Math.}
\bvolume{54}(\bissue{1}),
\bfpage{81}--\blpage{116}
(\byear{1930}).
doi:\doiurl{10.1007/BF02547518}
\end{barticle}
%
\endOrigBibText
\bptok{structpyb}%
\endbibitem

\bibitem{Lecam1986}
%
\begin{bbook}
\bauthor{\bsnm{Le~Cam}, \binits{L.}}:
\bbtitle{Asymptotic Methods in Statistical Decision Theory}.
\bpublisher{Springer}
(\byear{1986}).
\bid{doi={10.1007/978-1-4612-4946-7},mr={0856411}}
\end{bbook}
%
\OrigBibText
%
\begin{bbook}
\bauthor{\bsnm{Le~Cam}, \binits{L.}}:
\bbtitle{Asymptotic Methods in Statistical Decision Theory}.
\bpublisher{Springer}
(\byear{1986}).
doi:\doiurl{10.1007/978-1-4612-4946-7}
\end{bbook}
%
\endOrigBibText
\bptok{structpyb}%
\endbibitem

\bibitem{LehmannRomano2005}
%
\begin{bbook}
\bauthor{\bsnm{Lehmann}, \binits{E.L.}},
\bauthor{\bsnm{Romano}, \binits{J.P.}}:
\bbtitle{Testing Statistical Hypotheses},
\bedition{3}rd edn.
\bpublisher{Springer}
(\byear{2005}).
\bid{mr={2135927}}
\end{bbook}
%
\OrigBibText
%
\begin{bbook}
\bauthor{\bsnm{Lehmann}, \binits{E.L.}},
\bauthor{\bsnm{Romano}, \binits{J.P.}}:
\bbtitle{Testing Statistical Hypotheses},
\bedition{3}rd edn.
\bpublisher{Springer}
(\byear{2005}).
\MR{2135927}
\end{bbook}
%
\endOrigBibText
\bptok{structpyb}%
\endbibitem

\bibitem{LeskelaVihola2013}
%
\begin{barticle}
\bauthor{\bsnm{{Leskel\"a}}, \binits{L.}},
\bauthor{\bsnm{{Vihola}}, \binits{M.}}:
\batitle{{Stochastic order characterization of uniform integrability and
tightness.}}
\bjtitle{{Stat. Probab. Lett.}}
\bvolume{83}(\bissue{1}),
\bfpage{382}--\blpage{389}
(\byear{2013}).
\bid{doi={10.1016/j.spl.2012.09.\\023}}
\end{barticle}
%
\OrigBibText
%
\begin{barticle}
\bauthor{\bsnm{{Leskel\"a}}, \binits{L.}},
\bauthor{\bsnm{{Vihola}}, \binits{M.}}:
\batitle{{Stochastic order characterization of uniform integrability and
tightness.}}
\bjtitle{{Stat. Probab. Lett.}}
\bvolume{83}(\bissue{1}),
\bfpage{382}--\blpage{389}
(\byear{2013}).
doi:\doiurl{10.1016/j.spl.2012.09.023}
\end{barticle}
%
\endOrigBibText
\bptok{structpyb}%
\endbibitem

\bibitem{McCann1995}
%
\begin{barticle}
\bauthor{\bsnm{McCann}, \binits{R.J.}}:
\batitle{Existence and uniqueness of monotone measure-preserving maps}.
\bjtitle{Duke Math. J.}
\bvolume{80}(\bissue{2}),
\bfpage{309}--\blpage{323}
(\byear{1995}).
\bid{doi={10.1215/S0012-7094-95-08013-2}}
\end{barticle}
%
\OrigBibText
%
\begin{barticle}
\bauthor{\bsnm{McCann}, \binits{R.J.}}:
\batitle{Existence and uniqueness of monotone measure-preserving maps}.
\bjtitle{Duke Math. J.}
\bvolume{80}(\bissue{2}),
\bfpage{309}--\blpage{323}
(\byear{1995}).
doi:\doiurl{10.1215/S0012-7094-95-08013-2}
\end{barticle}
%
\endOrigBibText
\bptok{structpyb}%
\endbibitem

\bibitem{Muller1996}
%
\begin{barticle}
\bauthor{\bsnm{M{\"u}ller}, \binits{A.}}:
\batitle{Orderings of risks: A comparative study via stop-loss transforms}.
\bjtitle{Insurance: Mathematics and Economics}
\bvolume{17}(\bissue{3}),
\bfpage{215}--\blpage{222}
(\byear{1996}).
\bid{doi={10.1016/0167-\\6687(96)90002-5}}
\end{barticle}
%
\OrigBibText
%
\begin{barticle}
\bauthor{\bsnm{M{\"u}ller}, \binits{A.}}:
\batitle{Orderings of risks: A comparative study via stop-loss transforms}.
\bjtitle{Insurance: Mathematics and Economics}
\bvolume{17}(\bissue{3}),
\bfpage{215}--\blpage{222}
(\byear{1996}).
doi:\doiurl{10.1016/0167-6687(96)90002-5}
\end{barticle}
%
\endOrigBibText
\bptok{structpyb}%
\endbibitem

\bibitem{MullerStoyan2002}
%
\begin{bbook}
\bauthor{\bsnm{M\"uller}, \binits{A.}},
\bauthor{\bsnm{Stoyan}, \binits{D.}}:
\bbtitle{Comparison Methods for Stochastic Models and Risks}.
\bpublisher{John Wiley \& Sons, Ltd., Chichester}
(\byear{2002}).
\bid{mr={1889865}}
\end{bbook}
%
\OrigBibText
%
\begin{bbook}
\bauthor{\bsnm{M\"uller}, \binits{A.}},
\bauthor{\bsnm{Stoyan}, \binits{D.}}:
\bbtitle{Comparison Methods for Stochastic Models and Risks}.
\bpublisher{John Wiley \& Sons, Ltd., Chichester}
(\byear{2002}).
\MR{1889865}
\end{bbook}
%
\endOrigBibText
\bptok{structpyb}%
\endbibitem

\bibitem{OgryczakRuszczynske2002}
%
\begin{barticle}
\bauthor{\bsnm{Ogryczak}, \binits{W.}},
\bauthor{\bsnm{Ruszczy\'nski}, \binits{A.}}:
\batitle{Dual stochastic dominance and related mean-risk models}.
\bjtitle{SIAM J. Optim.}
\bvolume{13}(\bissue{1}),
\bfpage{60}--\blpage{78}
(\byear{2002}).
\bid{doi={10.1137/S1052623400375075}}
\end{barticle}
%
\OrigBibText
%
\begin{barticle}
\bauthor{\bsnm{Ogryczak}, \binits{W.}},
\bauthor{\bsnm{Ruszczy\'nski}, \binits{A.}}:
\batitle{Dual stochastic dominance and related mean-risk models}.
\bjtitle{SIAM J. Optim.}
\bvolume{13}(\bissue{1}),
\bfpage{60}--\blpage{78}
(\byear{2002}).
doi:\doiurl{10.1137/S1052623400375075}
\end{barticle}
%
\endOrigBibText
\bptok{structpyb}%
\endbibitem

\bibitem{RockafellarUryasev2000}
%
\begin{barticle}
\bauthor{\bsnm{Rockafellar}, \binits{R.T.}},
\bauthor{\bsnm{Uryasev}, \binits{S.}}:
\batitle{Optimization of conditional value-at-risk}.
\bjtitle{Journal of Risk}
\bvolume{2},
\bfpage{21}--\blpage{42}
(\byear{2000})
\end{barticle}
%
\OrigBibText
%
\begin{barticle}
\bauthor{\bsnm{Rockafellar}, \binits{R.T.}},
\bauthor{\bsnm{Uryasev}, \binits{S.}}:
\batitle{Optimization of conditional value-at-risk}.
\bjtitle{Journal of Risk}
\bvolume{2},
\bfpage{21}--\blpage{42}
(\byear{2000})
\end{barticle}
%
\endOrigBibText
\bptok{structpyb}%
\endbibitem

\bibitem{RockafellarUryasev2002}
%
\begin{barticle}
\bauthor{\bsnm{Rockafellar}, \binits{R.T.}},
\bauthor{\bsnm{Uryasev}, \binits{S.}}:
\batitle{Conditional value-at-risk for general loss distributions}.
\bjtitle{Journal of Banking \& Finance}
\bvolume{26}(\bissue{7}),
\bfpage{1443}--\blpage{1471}
(\byear{2002})
\end{barticle}
%
\OrigBibText
%
\begin{barticle}
\bauthor{\bsnm{Rockafellar}, \binits{R.T.}},
\bauthor{\bsnm{Uryasev}, \binits{S.}}:
\batitle{Conditional value-at-risk for general loss distributions}.
\bjtitle{Journal of Banking \& Finance}
\bvolume{26}(\bissue{7}),
\bfpage{1443}--\blpage{1471}
(\byear{2002})
\end{barticle}
%
\endOrigBibText
\bptok{structpyb}%
\endbibitem

\bibitem{Ruschendorf2013}
%
\begin{bbook}
\bauthor{\bsnm{R\"uschendorf}, \binits{L.}}:
\bbtitle{Mathematical Risk Analysis: Dependence, Risk Bounds, Optimal
Allocations and Portfolios}.
\bpublisher{Springer}
(\byear{2013}).
\bid{doi={10.1007/978-3-642-33590-7}}
\end{bbook}
%
\OrigBibText
%
\begin{bbook}
\bauthor{\bsnm{R\"uschendorf}, \binits{L.}}:
\bbtitle{Mathematical Risk Analysis: Dependence, Risk Bounds, Optimal
Allocations and Portfolios}.
\bpublisher{Springer}
(\byear{2013}).
\bid{doi={10.1007/978-3-642-33590-7}}
\end{bbook}
%
\endOrigBibText
\bptok{structpyb}%
\endbibitem

\bibitem{ShiryaevSpokoyniy2000}
%
\begin{bbook}
\bauthor{\bsnm{Shiryaev}, \binits{A.N.}},
\bauthor{\bsnm{Spokoiny}, \binits{V.G.}}:
\bbtitle{Statistical Experiments and Decisions: Asymptotic Theory}.
\bpublisher{World Scientific Publishing Co., Inc., River Edge, NJ}
(\byear{2000}).
\bid{doi={10.1142/\\9789812779243}}
\end{bbook}
%
\OrigBibText
%
\begin{bbook}
\bauthor{\bsnm{Shiryaev}, \binits{A.N.}},
\bauthor{\bsnm{Spokoiny}, \binits{V.G.}}:
\bbtitle{Statistical Experiments and Decisions: Asymptotic Theory}.
\bpublisher{World Scientific Publishing Co., Inc., River Edge, NJ}
(\byear{2000}).
doi:\doiurl{10.1142/9789812779243}
\end{bbook}
%
\endOrigBibText
\bptok{structpyb}%
\endbibitem

\bibitem{Skorokhod:65}
%
\begin{bbook}
\bauthor{\bsnm{Skorokhod}, \binits{A.V.}}:
\bbtitle{Studies in the Theory of Random Processes}.
\bpublisher{Addison-Wesley Publishing Co., Inc., Reading, Mass.}
(\byear{1965}).
\bid{mr={0185620}}
\end{bbook}
%
\OrigBibText
%
\begin{bbook}
\bauthor{\bsnm{Skorokhod}, \binits{A.V.}}:
\bbtitle{Studies in the Theory of Random Processes}.
\bpublisher{Addison-Wesley Publishing Co., Inc., Reading, Mass.}
(\byear{1965}).
\MR{0185620 (32 \#3082b)}
\end{bbook}
%
\endOrigBibText
\bptok{structpyb}%
\endbibitem

\bibitem{Strasser1985}
%
\begin{bbook}
\bauthor{\bsnm{Strasser}, \binits{H.}}:
\bbtitle{Mathematical Theory of Statistics: Statistical Experiments and
Asymptotic Decision Theory}.
\bpublisher{Walter de Gruyter \& Co., Berlin}
(\byear{1985}).
\bid{doi={10.1515/9783110850826}}
\end{bbook}
%
\OrigBibText
%
\begin{bbook}
\bauthor{\bsnm{Strasser}, \binits{H.}}:
\bbtitle{Mathematical Theory of Statistics: Statistical Experiments and
Asymptotic Decision Theory}.
\bpublisher{Walter de Gruyter \& Co., Berlin}
(\byear{1985}).
doi:\doiurl{10.1515/9783110850826}
\end{bbook}
%
\endOrigBibText
\bptok{structpyb}%
\endbibitem

\bibitem{Torgersen1991}
%
\begin{bbook}
\bauthor{\bsnm{Torgersen}, \binits{E.}}:
\bbtitle{Comparison of Statistical Experiments}.
\bpublisher{Cambridge University Press, Cambridge}
(\byear{1991}).
\bid{doi={10.1017/CBO9780511666353}}
\end{bbook}
%
\OrigBibText
%
\begin{bbook}
\bauthor{\bsnm{Torgersen}, \binits{E.}}:
\bbtitle{Comparison of Statistical Experiments}.
\bpublisher{Cambridge University Press, Cambridge}
(\byear{1991}).
doi:\doiurl{10.1017/CBO9780511666353}
\end{bbook}
%
\endOrigBibText
\bptok{structpyb}%
\endbibitem

\end{thebibliography}

%
\end{document}